\documentclass[11pt, a4paper, oneside, xcolor=dvipsnames]{amsart}

\usepackage{amssymb,amsmath, amsopn, amsthm}
\usepackage{mathtools}
\usepackage{hyperref}
\usepackage[usenames,dvipsnames]{xcolor}
\usepackage{pgf,tikz}
\usetikzlibrary{calc}
\usetikzlibrary{intersections, patterns}
\usepackage{paralist}
\usepackage{bbm}
\usepackage[T1]{fontenc}
\usepackage{lmodern}
\usepackage{comment}
\usepackage{tabularx}

\newcommand{\bs}{\symbol{92}}

\def\CC{{\mathbb C}}
\def\RR{{\mathbb R}}

\def\NN{{\mathbb N}}

\def\PPP{{\mathbb P}}

\def\ra{\rightarrow}

\def\lora{\longrightarrow}

\def\<{\langle}
\def\>{\rangle}

\def\norm{\lVert \cdot \rVert}
\def\B°{B^{\circ}}
\def\Bo{B^{\circ}}
\def\°{\circ}

\def\F°{F^{\circ}}
\def\subset{\subseteq}

\def\del{\partial}
\def\|{\text{ }|\text{ }}

\def\vec{\begin{pmatrix}}          
\def\tor{\end{pmatrix}}            

\def\w3{\sqrt{3}}                  

\def\gg{\mathfrak{g}}
\def\kk{\mathfrak{k}}
\def\pp{\mathfrak{p}}

\def\aa{\mathfrak{a}}

\def\sl{\mathfrak{sl}}

\def\H{\mathcal{H}}
\def\P{\mathcal{P}}

\def\C{\mathcal{C}}
\def\N{\mathcal{N}}

\def\E{\mathcal{E}}
\def\F{\mathcal{F}}
\def\Hn{\H_n}
\def\Pn{\P_n}
\def\PHn{\text{P}(\H_n)}

\def\delrel{\del_{\text{rel}}}

\def\sp4{\mathfrak{sp}(4,\RR)}
\def\Sp4{Sp(4,\RR)}  

\def\coloneqq{\mathrel{\mathop:}=}

 \def\xsat{\overline{X}^S_\tau}

 \def\xhor{\overline{X}^{\text{hor}}}
 
 \def\PSat{\overline{\Pn}^S}
 \def\abar+{\overline{\aa^+}}

\DeclareMathOperator{\cl}{cl}

\DeclareMathOperator{\ri}{ri}
\DeclareMathOperator{\conv}{conv}

\DeclareMathOperator{\PSL}{PSL}

\DeclareMathOperator{\SL}{SL}
\DeclareMathOperator{\SU}{SU}

\DeclareMathOperator{\PSU}{PSU}

\DeclareMathOperator{\ad}{ad}
\DeclareMathOperator{\Ad}{Ad}

\DeclareMathOperator{\diag}{diag}

 \newcommand\rquotient[2]{
        \mathchoice
            {
                \text{\raise1ex\hbox{$#1$}\Big/\lower1ex\hbox{$#2$}}%
            }
            {
                #1\,/\,#2
            }
            {
                #1\,/\,#2
            }
            {
                #1\,/\,#2
            }
    }
    
\newcommand\lquotient[2]{
        \mathchoice
            {
                \text{\lower1ex\hbox{$#1$}$\Big\backslash$\raise1ex\hbox{$#2$}}%
            }
            {
                _#1\,\backslash\,^#2
            }
            {
                _#1\,\backslash\,^#2
            }
            {
                _#1\,\backslash\,^#2
            }
    }
\newcommand{\frakg}{\mathfrak{g}}
\newcommand{\fraka}{\mathfrak{a}}
\newcommand{\frakk}{\mathfrak{k}}

\newtheoremstyle{standard}
{8pt} 
{8pt} 
{\slshape}    
{} 
{\bfseries}  
{}           
{0.5em}        
{\thmname{#1}\thmnumber{ #2}\thmnote{ (#3)}}
 
\newtheoremstyle{standarddefi}
{8pt} 
{8pt} 
{}    
{} 
{\bfseries}  
{}           
{0.5em}        
{\thmname{#1}\thmnumber{ #2}\thmnote{ (#3)}}

\newtheoremstyle{nonumber}
{8pt} 
{8pt} 
{\slshape}    
{} 
{\bfseries}  
{}           
{0.5em}        
{\thmname{#1}\thmnote{ #3} }

\newtheoremstyle{nonumberprf}
{8pt} 
{8pt} 
{}    
{} 
{\bfseries}  
{}           
{0.5em}        
{\thmname{#1}\thmnote{ #3} }
 
\newtheoremstyle{comments}
{8pt} 
{8pt} 
{\footnotesize  }    
{} 
{\bfseries}  
{}           
{0.5em}        
{\thmname{#1}\thmnote{ #3} }

\theoremstyle{standard}
\newtheorem{thrm}{Theorem}[section]
\newtheorem{prop}[thrm]{Proposition}
\newtheorem{lem}[thrm]{Lemma}

\theoremstyle{standarddefi}
\newtheorem{defi}[thrm]{Definition}
\newtheorem{ex}[thrm]{Example}
\newtheorem{rem}[thrm]{Remark}

\newtheorem{quest}[thrm]{Question}

\theoremstyle{nonumber}

\theoremstyle{nonumberprf}

\newtheorem{notation}{Notation}

\allowdisplaybreaks

\begin{document}

\title{Horofunction Compactifications of Symmetric Spaces}

\author{Thomas Haettel}
\thanks{}

\author{Anna-Sofie Schilling}
\thanks{AS was partially supported by the European Research Council under ERC-Consolidator grant 614733, and by the German Research Foundation in the RTG 2229 Asymptotic Invariants and Limits of Groups and Spaces}

\author{Cormac Walsh}
\thanks{}

\author{Anna Wienhard}
\thanks{AW was partially supported by the National Science Foundation under agreement DMS-1536017, by the Sloan Foundation, by the European Research Council under ERC-Consolidator grant 614733, and by the Klaus-Tschira-Foundation.}

\date{\today}

\begin{abstract}
	\noindent We consider horofunction compactifications of symmetric spaces with respect to invariant Finsler metrics. We show that any (generalized) Satake compactification can be realized as a horofunction compactification with respect to a polyhedral Finsler metric.
\end{abstract}

\maketitle
\thispagestyle{empty}

\bibliographystyle{alpha}

 \tableofcontents


\section{Introduction}
Symmetric spaces of non-compact type arise in many areas of mathematics. Topologically they are diffeomorphic to a finite-dimensional vector space, in particular they are non-compact. The problem of constructing compactifications of symmetric spaces of non-compact type has been a classical problem. For an overview of compactifications of symmetric spaces see \cite{BJ, GJT}. 

Any proper metric space $(X,d)$ can be compactified by embedding $X$ into the space of continuous functions $C_{p_0}(X)$ on $X$ which vanish at a fixed base point $p_0$:
\begin{align} 
      \begin{split}
	  X &\lora C_{p_0}(X)  \\
	  z &\longmapsto d(\cdot,z) - d(p_0, z). 
      \end{split}   
\end{align}
The closure of the image is the horofunction compactification $\overline{X}_d^{hor}$ of $(X,d)$. 

In this article we investigate horofunction compactifications of symmetric spaces endowed with invariant Finsler metrics.

It is well known that the visual compactification of a symmetric space $X$ is realized as horofunction compactification with respect to the invariant Riemannian metric. 
We show that all Satake compactifications of $X$, and more generally all generalized Satake compactifications as defined in \cite{GuichardKasselWienhard} can be realized as horofunction compactifications with respect to polyhedral Finsler metrics on $X$. 

Any $G$-invariant Finsler metric on $X$ induces a Weyl group invariant norm on a maximal flat $F \cong \RR^k$. The Finsler metric is said to be polyhedral, if the unit ball for its Weyl group invariant norm on $F \cong \RR^k$ is a finite sided polytope. 

Before stating the result more precisely we recall that Satake compactifications $\overline{X}_\tau^S$ are associated to irreducible faithful representations 
$\tau: G \ra \PSL(n,\CC)$, which give rise to embeddings $X = G/K \rightarrow \PPP(Herm(\CC^n))$, $gK \mapsto [\tau(g)^*\tau(g)]$, see \cite{Sat}. Generalized Satake compactifications are defined the same way, but allowing $\tau$ to be reducible. There are finitely many isomorphism classes of Satake compactifications, determined by subsets of the set of simple roots, but infinitely many isomorphism classes of generalized Satake compactifications. 

\begin{thrm}\label{thm:satake}
Let $X = G/K$ be a symmetric space of non-compact type. Any generalized Satake compactification of $X$ can be realized as the horofunction compactification of a  polyhedral $G$-invariant Finsler metric on $X$. 
More precisely, if the generalized Satake compactification is given by a representation $\tau: G \to \PSL(n,\CC)$, then $\overline{X}_\tau^S$ is isomorphic to $\overline{X}_d^{hor}$, where $d$ is the Finsler metric on $X$ whose unit ball in a maximal Cartan subspace $\aa$  is the polytope dual to $-D = - \conv(\mu_1, \cdots , \mu_k)$, where $\mu_i$ are the weights of the representation $\tau$. 
\end{thrm}

\begin{rem}
The idea to realize Satake compactifications as horofunction compactifications with respect to polyhedral Finsler metrics has been sketched in the second authors diploma thesis \cite{Schilling}. Specific Satake compactifications have been realized as horofunction compactifications of Finsler metrics before. Friedland and Freitas \cite{FriedFreiI, FriedFreiII} describe the horofunction compactification for Finsler $p$-metrics on $GL(n, \CC)/U_n$ for $p \in [1, \infty]$, which they show to agree with the visual compactification for $ p>1$, and the horofunction compactification of the Siegel upper half plane of rank $n$ for the $l_1$-metric, which they show to agrees with the bounded symmetric domain compactification, a minimal Satake compactification. Kapovich and Leeb realize the maximal Satake compactification of a symmetric space $X = G/K$ of non-compact type as the horofunction boundary with respect to a $G$-invariant Finsler metric on $X$ \cite{KL}. Parreau \cite{Parreau} shows that the horofunction 
compactification with respect to 
the Weyl chamber valued distance function is isomorphic to the maximal Satake compactification. There are related constructions for buildings, too, see for example  \cite{Brill}. Another paper by Ciobotaru, Kramer and Schwer \cite{CKS} on horofunction compactifications is in preparation.  
\end{rem}

In order to describe the horofunction compactification $\overline{X}_d^{hor}$ in more detail, we will make use of the Cartan decomposition $G = KAK$. The key result  is then the following theorem that the closure of a maximal flat $F = A p_0 \subset X$ in $\overline{X}_d^{hor}$ is isomorphic to the horofunction compactification of $F \cong \RR^k$ with respect to the induced Weyl group invariant Finsler norm. We formulate it here for a polyhedral Finsler metric, but prove it in the more general setting Finsler metrics satisfying a mild technical condition (the Convexity Lemma~\ref{lem:convexity lemma}).
The horofunction compactification of $F \cong \RR^k$ with respect to a Weyl group invariant polyhedral Finsler norm has been determined in \cite{WalII, JS}. 

\begin{thrm}\label{thm:flat_compactification}
Let $X = G/K$ be a symmetric space of non-compact type. Consider a polyhedral $G$-invariant Finsler metric on $X$ such that the set of extreme sets of the dual unit ball of a flat is closed. Let $\overline{X}^{hor}$ be the horofunction compactification of $X$ with respect to this Finsler metric. Then the closure of a maximal flat $F$ in $\overline{X}^{hor}$ is isomorphic to the horofunction compactification of $F$ with respect to the induced metric. 
\end{thrm}

Theorem~\ref{thm:flat_compactification} then allows us to prove Theorem~\ref{thm:satake} by comparing the closure of a maximal flat $F = A p_0 \subset X$ in $\overline{X}_d^{hor}$ with the closure  of $F$ in a generalized Satake compactification. The closure of a flat in a generalized Satake compactification is described in \cite{GuichardKasselWienhard}. 

In Section~\ref{sec:finsler}, we review the structure theory of symmetric spaces, and recall a characterization of $G$-invariant Finsler metrics. In Section~\ref{sec:horofunction}, we review the horofunction compactification of metric spaces, and focus on the case of normed vector spaces. Using work of the third author the  technical statement we need, the Convexity Lemma, for most norms. In Section~\ref{sec:flat to space}, we consider a $G$-invariant Finsler metric on the symmetric space $X$ which satisfies the Convexity Lemma. For each maximal flat $F$ in $X$, we prove that the closure of $F$ in the horofunction compactification of $X$ is isomorphic to the intrinsic horofunction compactification of $F$. So we reduce the study of horofunction compactifications of symmetric spaces to the study of horofunction compactifications of maximal flats. In Section~\ref{sec:satake}, we combine the previous results to prove that each (generalized) Satake 
compactification is a horofunction compactification for a specific polyhedral Finsler norm on the symmetric space.

\section{Invariant Finsler metrics on symmetric spaces} \label{sec:finsler}
In this section we first review the necessary structure theory of semisimple Lie groups, see \cite{Helg} for details,  and recall a characterization of $G$-invariant Finsler metrics due to Planche \cite{Plan}.

\subsection{Structure Theory} 

Throughout the article we denote by $G$ a real semisimple Lie group with finite center, and by $\frakg$ its Lie algebra. $K <G$ denotes a maximal compact subgroup, and $\frakk \subset \frakg$ its Lie algebra. 
The (Riemannian) symmetric space associated to $G$ is $X = G/K$, and $p_0 = eK$ denotes its base point. 

\subsubsection{Cartan decomposition} 
The Lie algebra of $G$ decomposes as 
\[
      \gg = \kk \oplus \pp,
\]
where $\pp$ is the orthogonal complement of $\kk$ with respect to the Killing form $\kappa$  on $\gg$. 

We fix a maximal abelian subalgebra $\aa \subset \pp$, and denote by $A= \exp (\aa)<G$ the corresponding connected subgroup of $G$. Then $ F= A\cdot p_0 \subset X$ is a maximal flat. 

All maximal abelian subalgebras are conjugate to each other, and  $\pp = \Ad(K) \aa = \bigcup_{k \in K}\Ad(k) \aa$.

Let $\Sigma = \Sigma(\gg, \aa)\subset \fraka^*$ denote the system of restricted roots, i.e. $\alpha \in \Sigma(\gg, \aa)$ if and only if 
\[ 
      \gg_{\alpha}:= \{ V\in \frakg\, |\, \ad(H)V = \alpha(H) V \ \forall H \in \aa\}
\]
is non-zero. 

For each $\alpha \in \Sigma$ consider the hyperplane $\ker(\alpha) \subset \aa$. Each of them divides the vector space $\aa$ into two half-spaces. The connected components of the set $\aa \setminus \bigcup_{\alpha \in \Sigma} \ker(\alpha)$ are called \emph{Weyl chambers}. We fix one of these chambers to be the positive Weyl chamber $\aa^{+}$, and define \emph{positive roots} by 
\[
      \Sigma^+ \coloneqq \{ \alpha \in \Sigma \| \alpha(H) > 0 \ \ \forall H \in \aa^+\}.
\]
We denote by  $\Delta$ the set of \emph{simple roots} :  
\[
      \Delta \coloneqq \{ \alpha \in \Sigma^+ \| \alpha(H) \text{ is not the sum of two positive roots}\}.
\]
The simple roots form a basis of $\Sigma$ in the sense that we can express every root as a linear combination of elements in $\Delta$ with integer coefficients which are either all $\geq 0$  or all $\leq 0$.

\begin{lem}[Cartan decomposition; \cite{Helg}, Thm.V.6.7 and Thm.IX.1.1] \label{cartan}
      Let $\aa^{+}$ be a positive Weyl chamber. Set $A^+ \coloneqq \exp(\aa^+) \subset G$ and denote by $\overline{A^+}$ its closure. Note that $\overline{A^+} = \exp (\overline{\aa^+})$.  
      For every element $g\in G$ there exist  $k_1, k_2 \in K$ and some $a \in \overline{A^+}$ such that $g = k_1 a k_2$. 

      We shortly write 
      \[
	  G = K\overline{A^+}K, 
      \]
      and call this a Cartan decomposition of $G$. 
\end{lem}

\subsubsection{The Weyl group}
Let $\C_K(\aa) = \{ k \in K \| \Ad(k)(H) = H \ \forall H \in \aa\}$ denote the centralizer of $\aa$ in $K$ and 
$\N_K(\aa) = \{ k \in K \| \Ad(k) \aa \subset \aa\}$ its normalizer. Then $\C_K(\aa) \trianglelefteq \N_K(\aa)$ is a normal subgroup. 

\begin{defi}
      The quotient 
      \[
	  W\coloneqq \rquotient{\N_K(\aa)}{\C_K(\aa)}
      \]
      is the \emph{Weyl group}. It acts simply transitively on the set of Weyl chambers. The Weyl group is generated by the reflections in the hyperplanes $\ker(\alpha)$ for $\alpha \in \Delta$ and can also be expressed as $\N_K(A)/\C_K(A)$. 
\end{defi}

\subsection{Finsler Geometry}

A Finsler metric on a smooth manifold $M$ generalizes the concept of  Riemannian metric. It is a continuous family of  (possibly asymmetric) norms on the tangent spaces, which are not necessarily induced by an inner product. 

\begin{defi}
      Let $M$ be a smooth manifold. A \emph{Finsler metric} on $M$ is a continuous function 
      \[
	  F: TM \lora [0, \infty)
      \]
      such that, for each $p \in M$, the restriction $F|_{T_pM} : T_pM \lora [0,\infty)$ is a (possibly asymmetric) norm.
\end{defi}

The length and (forward) distance on a Finsler manifold can be defined in the same way as on a Riemannian manifold: 

\begin{defi}
      The \emph{length} of a curve $\gamma: [0,1] \subset \RR \lora M$ is defined as 
      \[
	  L(\gamma) \coloneqq \int_I F(\gamma(t), \dot{\gamma}(t))dt. 
      \]
      The \emph{forward distance} between two points $p, q \in M$ is given by
      \[
	  d_F(p,q) \coloneqq \inf_{\gamma}  L(\gamma),
      \]
      where the infimum is taken over all piecewise continuously differentiable curves $\gamma: [0,1] \lora M$ with $\gamma(0) = p$ and $\gamma(1) = q$.
\end{defi}

\begin{rem}
      As the homogeneity in the definition for a Finsler metric only holds for positive scalars, the norms on the tangent spaces do not have to be symmetric.  Therefore in general $d_F(p,q) \neq d_F(q,p)$. 
\end{rem}

The symmetric space $X$ carries a $G$-invariant Riemannian metric, which is essentially unique (up to scaling on the irreducible factors). However, $X$ carries many $G$-invariant Finsler metrics. Such $G$-invariant Finsler metrics on $X$ and their isometry groups have been investigated by Planche, who proved:  

\begin{thrm}[\cite{Plan}, Thm.6.2.1] \label{planche}
      There is a bijection between 
      \begin{itemize}
	  \item[i)] the $W$-invariant convex closed balls $B$ of $\aa$, 
	  \item[ii)] the $\Ad(K)$-invariant convex closed balls $C$ of $\pp$, 
	  \item[iii)] the $G$-invariant Finsler metrics on $X$.
      \end{itemize}
\end{thrm}

In particular, any $G$-invariant Finsler metric on $X$ gives rise to a (not necessarily symmetric) norm on the vector space $\aa$, whose unit ball is the $W$-invariant convex ball $B$, and it is in turn completely determined by this norm. 
\begin{defi}
      A $G$-invariant Finsler metric on $X$ is said to be \emph{polyhedral} if its $W$-invariant convex ball $B$ in $\aa$ is a finite sided polytope. 
\end{defi}

\section{Horofunction compactifications} \label{sec:horofunction}

The horofunction compactifications of normed vector spaces have been described by Walsh in \cite{WalI}. We will not give a full description of his work, but focus on the case when the convex ball $B$ is a finite sided polytope. In this setting the horofunction compactification has been described in detail by Ji and Schilling in \cite{JS}, see also \cite{KarlssonMetzNoskov} for a description of horoballs. 

The key results for us are a characterization of converging sequences, a convexity lemma, and the identification between the horofunction compactification and the dual convex polytope $B^\circ$.

\subsection{The horofunction compactification of a metric space} \label{subsec:defi_of_horofunction}

 Let $(X,d)$ be a metric space whose metric is possibly non-symmetric and with its topology induced by the symmetrized distance 
 \[
      d_{sym}(x,y) \coloneqq d(x,y) + d(y,x)
 \] 
 for all $x,y \in X$. Let $C(X)$ be endowed with the topology of uniform convergence on bounded sets with respect  to $d_{sym}$. Fix a basepoint $p_0 \in X$. Let $C_{p_0}(X)$ be the set of continuous functions on $X$ which vanish at $p_0$. This space is homeomorphic to the quotient of $C(X)$ by constant functions, $\widetilde{C}(X) \coloneqq \rquotient{C(X)}{const}$. Define the map
\begin{align} \label{Psi}
      \begin{split}
	  \psi: X &\lora \widetilde{C}(X) \\
	  z &\longmapsto \psi_z
      \end{split}
 \end{align}
 where 
 \[
      \psi_z(x) = d(x,z) - d(p_0, z)
 \]
 for all $x \in X$. Then $\psi$ is continuous and injective. If $X$ is geodesic, proper with respect to $d_{sym}$ and if $d$ is symmetric with respect to convergence, that is,  $d(x_n,x) \ra 0 \ \text{ iff } \ d(x, x_n) \ra 0$ for any sequence $(x_n)_{n \in \NN}$ and some $x \in X$, then the closed set $\cl\{\psi_z \| z \in X\}$ is compact and $\psi$ is an embedding of $X$ into $\widetilde{C}(X)$. For more details see \cite[p.4 and Prop. 2.2]{WalI}. 

\begin{defi}
      The \emph{horofunction boundary} $\del_{hor}(X)$ of $X$ in $\widetilde{C}(X)$ is defined as 
      \[
	  \del_{hor} X \coloneqq \left ( \cl\{\psi_z \| z \in X\} \right) \setminus \{\psi_z \| z \in X\}.
      \]
      Its elements are called \emph{horofunctions}. If $\cl\{\psi_z \| z \in X\}$ is compact, then the set \mbox{$\xhor \coloneqq \cl\{\psi_z \| z \in X\} = X \cup \del_{hor} X$} is called the \emph{horofunction compactification} of $X$. 
\end{defi}

\begin{rem}\label{basepoint}
      The definition of $\psi_z$ and therefore also those of $\psi$ and  $\del_{hor} X$ depend on the choice of the basepoint $p_0$. One can show by a short calculation that, if we choose an alternative basepoint, the corresponding boundaries are homeomorphic. 
\end{rem}

\subsection{Horofunction compactifications of normed vector spaces}

Let in the following $(V, \lVert \cdot \rVert)$ always denote a finite-dimensional normed space and let $\langle \cdot | \cdot \rangle$ denote the dual pairing on it. 
\begin{notation}
      For a subset $C \subset V$ we denote by $V(C) \subset V$ the subspace generated by $C$ and by $V(C)^\bot$ its orthogonal complement, for some fixed arbitrary Euclidean structure on $V$. The orthogonal projection to $V(C)$ will be denoted by $\Pi_C$. For an element $v \in V$ we write $v = v_C + v^C$ with $v_C \in V(C)$ and $v^C \in V(C)^\bot$.
\end{notation}
 As the norm might be asymmetric, note that we use the convention 
      \[
	  d(x,z) = \lVert z - x \rVert. 
      \]

\begin{defi}
      Let $B$ be the unit ball with respect to the norm $\lVert \cdot \rVert$. The \emph{dual unit ball} $\Bo$ of $B$ is defined to be the polar of $B$:
      \[
	  \Bo \coloneqq \{ y \in V^* \| \langle y | x \rangle \geq -1 \ \ \forall x \in B \},
      \]
      where $V^*$ denotes the dual space of $V$. 
\end{defi}

\begin{rem}
      Some authors define the polar and therefore the dual unit ball by the condition $\langle y | x \rangle \leq 1 \ \ \forall x \in B$. As long as $B$ is symmetric, this makes as a set no difference. 
\end{rem}

\begin{defi}
A convex subset $E$ of a convex set $C$ is called \emph{extreme} if the endpoints of any line segment in $C$ are contained in $E$ whenever any interior point of the line segment is. 
\end{defi}

\begin{defi}
      The \emph{relative interior} $\ri S$ of a set $S \subset V$  is the interior of $S$ when $S$ is seen as a subset of a minimal affine subspace of $V$ containing $S$.   
      Similary, the \emph{relative boundary} $\delrel S$ of $S$ is the boundary of $S$ seen in the minimal affine subspace containing $S$.
\end{defi}

In the following we will consider norms $\norm$ which have a polyhedral unit ball $B$, that is, $B$ is a polytope. We will always assume our polytope to be finite, bounded and convex. The extreme sets of a polytope are its faces. Such a polytope can be described in two ways: either as the convex hull of finitely many points or as the intersection of finitely many halfspaces. The interplay of them is strongly related to the relation between the unit ball $B$ and its dual $\Bo$: 

\begin{rem} \label{rem:Bo}
      Let the unit ball $B \subset V$ be given as the convex hull of a finite set of points, $B = \conv\{a_1, \ldots, a_k\}$. In this description we want all points $a_i$ to be extremal, i.e. $\conv\{a_1, \ldots, a_k \} \neq \conv\{a_1, \ldots, \overset{\wedge}{a_j}, \ldots, a_k\}$ for all $j \in \{1, \ldots, k\}$. Then each vertex $a_i \in V$ determines a halfspace $V_i \coloneqq  \{ y \in V^* \| \langle y | a_i \rangle \geq -1 \} \subset V^*$ which contains the origin in its interior. The boundary $H_i$ of such a halfspace is a hyperplane for which it holds $\langle H_i | a_i \rangle = -1$, that is, $\langle h_i | a_i \rangle  = -1 $ for all $h_i \in H_i$. Then the dual unit ball $B^\circ$ is given by:
      \begin{align*} 
	  B^\circ &= \bigcap_{i = 1}^k V_i \\
	  &= \{ y \in V^* \| \langle y | a_i \rangle \geq -1 \ \ \forall i = 1, \ldots, k\}.
      \end{align*}
      As the unit ball $B$ is closed convex and contains the origin as an interior point, we know by the theory of polars and convex sets that 
      \[
	  (B^\circ)^\circ = B.
      \]
      Therefore, if $B$ is given as the intersection of halfspaces, with this result we can easily determine $B^\circ$ as the convex hull of a set of points. Starting from this description, we only want to consider relevant halfspaces in the intersection, that is, $B \cap \del V_i$ is a $m-1$ dimensional face\footnote{A \emph{k-face} of a polytope $P = \conv\{p_1, \ldots, p_r\}$ is the $k$-dimensional intersection of $P$ with one or more hyperplanes $H_i$ ($i \in \{1, \ldots, r\}$) defining $P$. } of $B$, which we will also call a \emph{facet}.
\end{rem}

Based on these two descriptions there is a one-to-one correspondence between the faces of $B$ and those of $\Bo$: 

\begin{lem}
      Let $B \subset V$ be a polyhedral unit ball and $\Bo \subset V^*$ its dual. For a face $F \subset B$ define its \emph{dual set} by
      \[
	  F^\circ \coloneqq \{ y \in \Bo \| \langle y | x \rangle = -1 \ \forall x \in F\} \subset \Bo.
      \]
      Then $F^\circ$ is a face of $\Bo$ and it holds 
      \[ 
	  \dim F + \dim F^\circ = m-1.
      \]
\end{lem}

\begin{proof}
      To show that $F^\circ$ is a face of $\Bo$ we have to show that it is an extreme set of $\Bo$. Recall that $F^\circ \subset \Bo$ is an extreme set, if some interior point of a line in $\Bo$ lies in $F^\circ$, then also both endpoints of the line. Therefore let $y \in F^\circ$ and $y_1, y_2 \in \Bo$ such that $y = \frac{1}{2}(y_1 + y_2) \in F^\circ$. For any $x \in F$ we have 
      \begin{align*}
	  -1 = \langle y | x \rangle = \frac{1}{2}(\langle y_1 | x \rangle + \langle y_2 | x \rangle) \geq -1,
      \end{align*}
      as both $y_1, y_2 \in \Bo$. Equality holds if and only if $\langle y_1 | x \rangle = \langle y_2 | x \rangle = -1$ and therefore $y_1, y_2 \in F^\circ$.
      
      We now show the formula for the dimensions. Let $V(F) \subset V$ be the subspace generated by $F$. We show that the dual $(V(F)^\bot)^* \subset V^*$ of its orthogonal complement is parallel to the affine subspace generated by $F^\circ$. As $F \subset B$ is part of a hyperplane defining $B$, we can find $\tilde{z} \in V^*$ and $r \in \RR$ such that $F = \{ y \in V | \langle \tilde{z} | y \rangle = r\} \cap B$ and $\langle \tilde{z} | y \rangle > r$ for all $ y \in B \setminus F$. As the origin is in the interior of $B$, $r < 0$. Set $z \coloneqq \frac{1}{|r|} \tilde{z}$, then $z \in F^\circ$.  
      We claim that 
      \[
	  (V(F)^\bot)^* = V(F^\circ - z).
      \]
      To see this, let $B = \conv\{a_1, \ldots, a_k\}$,  and let $S_F \subset \{1, \ldots, k\}$ be those indices such that  $F = \conv\{a_i | i \in S_F\}$. Let $y \in (V(F)^\bot)^*$ and $\varepsilon >0$.  Then $\langle z + \varepsilon y | a_i \rangle = -1$ for all $i \in S_F$ and $\langle z + \varepsilon y | a_j \rangle > -1 + \varepsilon \langle y | a_j \rangle$ for all $j \notin S_F$. As $B$ has only finitely many vertices $a_i$ ($i \in \{1, \ldots, k\}$), we can choose $\varepsilon$ small enough such that $\langle z + \varepsilon y | a_j \rangle > -1$ for all $j \notin S_F$. This implies $z + \varepsilon y \in F^\circ$  and therefore $y \in \frac{1}{\varepsilon}(F^\circ - z) \subset V(F^\circ - z)$. 
      The other inclusion follows immediately as $\langle x | y \rangle = -1$ for all $x \in F^\circ, y \in F$. \\
      With $\dim(F) = \dim(V(F)) - 1$ (as $0 \notin F$) and $\dim(V(F^\circ - z)) = \dim(F^\circ)$ (because $0 \in F^\circ-z$) we obtain 
      \begin{align*}
	  \dim(V) = \dim(V(F)) + \dim(V(F)^\bot) = \dim(F) + 1 + \dim(F^\circ),
      \end{align*} 
      which finishes the proof.
\end{proof}

In the case of a finite-dimensional normed space with polyhedral norm, Walsh gives a criterion (\cite[Thm.1.1 and Thm.1.2]{WalII}) to calculate the horofunctions  explicitely by using the Legendre-Fenchel transform of some special map. We rewrite these functions using some kind of pseudo-norm, see \cite[p.5]{WalII} or \cite{Schilling} for more details:

\begin{defi}
      Let $C \subset V^*$ be a convex set. For $p \in V$ we set 
      \[
	  |p|_C := -\inf_{q\in C} \langle q | p \rangle.
      \]  
\end{defi}

\begin{rem} \label{rem:notanorm}
      $\lvert \cdot \rvert_C$ is in general not a norm but 
      \[
	  \lvert \cdot \rvert_{\Bo} = \norm.
      \]

\end{rem}

Now we define the functions that will turn out to be the horofunctions in the horofunction compactification of $V$ with respect to the norm with unit ball $B$. Let $E \subset \Bo$ be a face of $\Bo$ and $p \in V(E^\circ)^\bot$ be a point. Then we define
\begin{align*}
      h_{E,p}: V &\lora \RR \\
      y &\longmapsto \lvert p - y \rvert_E - \lvert p \rvert_E.
\end{align*}
A short calculation shows that only the orthogonal part of $p$ makes a contribution: 
\[
      h_{E, p} = h_{E, p^F}, 
\]
with $F = E^\circ$ and $p^F  \in V(F)^\bot$. 
If we choose not a proper face $E$ but the entire dual unit ball, we get by Remark \ref{rem:notanorm} that 
\[
     \psi_z = h_{\Bo, z} 
\]
for all $z \in V$.

Combining the results of Walsh with some calculations that can be found in \cite{JS} we obtain

\begin{thrm}[\cite{WalI}, Thm.1.1 , \cite{JS}, p.10]
      Let $(V, \norm)$ be a finite-dimensional normed space where the unit ball $B$ is a polytope. Then the set of horofunctions is given as
      \[
	  \del_{hor}(V) = \{ h_{E,p} | E \subset B \text{ is a proper face and } p \in V(E^\circ)^\bot\}.
      \]
\end{thrm}

\begin{ex} \label{ex:viereck2}
 As an example let us consider $\RR^2$ equipped with the $L^1$-norm. For notations of the faces see Figure \ref{fig: example_B_Bo}. 

 \begin{figure}[h!]
	\centering
	\includegraphics{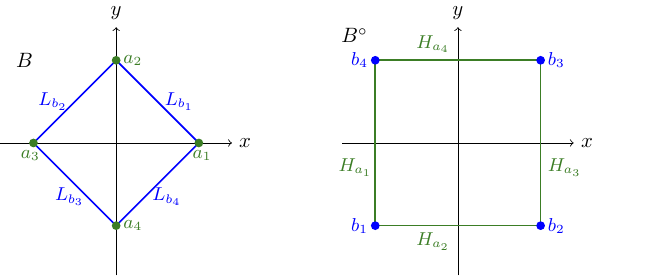}
	\caption{The unit ball $B$ and its dual $\Bo$}\label{fig: example_B_Bo}
\end{figure}

 Then its unit ball is given as the convex set 
\newpage
 \begin{align*}
      B &= \conv \big\{(1,0), (0,1), (-1,0), (0, -1) \big\} = \conv\{a_i | i = 1, \ldots, 4\} \\
      &= \bigcap_{i = 1}^4 \{ x \in \RR^2 \| \langle b_i | x \rangle  \geq -1 \}
 \end{align*}
 with $b_1 = (-1,-1), b_2 = (1, -1), b_3 = (1,1)$ and  $b_4 = (-1,1)$. 
 By Remark \ref{rem:Bo} the dual unit ball is given by
 \begin{align*}
      \Bo &= \bigcap_{i = 1}^4 \{y \in \RR^2 \| \langle y | a_i \rangle \geq -1 \} \\
      &= \conv\{b_1, b_2, b_3, b_4\}.      
 \end{align*}
 The faces of $B$ are $\F = \{ a_i, L_{b_i} | i = 1, \ldots, 4 \}$ where the facets are given by $L_{b_i} \coloneqq \{ x \in \RR^2 | \langle b_i | x \rangle = -1 \} \cap B$. Similary, the faces of $\Bo$ are \mbox{$\E = \{ b_j, H_{a_j} | j = 1, \ldots, 4 \}$} with $H_{a_j} \coloneqq \{ y \in \RR^2 | \langle y | a_j \rangle = -1 \} \cap \Bo$. \\
 As indicated by the notation, the dual faces are:
 \begin{align*}
      \{a_j\}^\circ &= H_{a_j}\\
      (L_{b_i})^\circ &= \{b_i\}.
 \end{align*}
The horofunctions then are for example given by 
 \begin{align*}
  h_{\{b_j\},p}(y)&= \lvert p - y \rvert_{\{b_j\}} - \lvert p \rvert_{\{b_j\}} = \langle b_j | y \rangle \quad \text{ for all } j = 1, \dots, 4 \\
  h_{H_{a_1}, p}(y) &= -y_1 + \lvert p_2 - y_2 \rvert - \lvert p_2 \rvert.
 \end{align*}
\end{ex}

\begin{rem}
 It is a general result that if $E \subset \Bo$ is a vertex, then $h_{E,p}$ is independent of the point $p$.  
\end{rem}

\subsection{Almost geodesics and Busemann points}

We will recall one of the main results of~\cite{WalII} about the horofunctions of a general normed vector space, not necessarily with a polyhedral norm.

Consider any metric space $(X,d)$ with a basepoint $b$.

\begin{defi}
	An \emph{almost geodesic} in a metric space $(X, d)$ is a sequence $x_n$ in $X$ such that $d(b, x_n)$ is unbounded and, for all $\epsilon>0$,
	\begin{align}
		\label{eqn:reverse_triangle_inequality}
		d(b, x_m) + d(x_m, x_n) < d(b, x_n) + \epsilon,
	\end{align}
	for $m$ and $n$ large enough, with $m \le n$.
\end{defi}

Note that this is a slight variation of the original definition by
Rieffel \cite{Rief}. The main difference is that his almost geodesics
were parameterized to be approximately unit speed.

\begin{defi}
	A \emph{Busemann point} is a horofunction in $\del_{hor}(X)$ that is the limit of some almost geodesic sequence in $X$.
\end{defi}

For normed vector spaces, Walsh obtains a very nice criterion to answer the question: when are all horofunctions Busemann points?

\begin{thrm}[\cite{WalII}, Thm.1.2] \label{thm:walsh}
	Consider any finite-dimensional normed vector space. Then every horofunction is a Busemann point if and only if the set of extreme sets of the dual unit ball is closed. 
\end{thrm}

In this Theorem, the topology on the set of extreme sets is the Chabauty topology on the space of all closed subspaces of the dual unit ball. If $X$ is any locally compact topological space, recall that the space $Sub(X)$ of all closed subspaces of $F$ is endowed with a natural compact topology called the Chabauty topology (see~\cite{Bourbaki_integration} for details). When $X$ is metrizable, then $Sub(X)$ is also metrizable, and a sequence a closed subspaces $(F_n)_{n \in \NN}$ converges to $F$ in $Sub(X)$ if:
\begin{itemize}
\item for any $x \in F$, for any $n \in \NN$, there exists $x_n \in F_n$ such that the sequence $(x_n)_{n \in \NN}$ converges to $x$, and
\item for any sequence $(x_n)_{n \in \NN}$ in $X$ such that for any $n \in \NN$ we have $x_n \in F_n$, every accumulation point of $(x_n)_{n \in \NN}$ belongs to $F$. 
\end{itemize}

\subsection{Convexity Lemma}\label{sec:convexity}

In this section, we will prove a technical convexity result, which will be used to determine the compactification of a flat.

\begin{lem}
\label{lem:reverse_triangle_inequality}
	Let $x_n$ be an almost geodesic in a metric space $(X, d)$.
	Then, for any $\epsilon>0$,
	\begin{align*}
		d(x_i, x_j) + d(x_j, x_k) < d(x_i, x_k) + \epsilon
	\end{align*}
	for $i$, $j$, and $k$ large enough, with $i \le j \le k$,
\end{lem}

\begin{proof}
	Applying~(\ref{eqn:reverse_triangle_inequality}) twice, first to the points
	$x_i$ and $x_j$, and then to the points $x_j$ and $x_k$, and then using the triangle inequality, we get, for $i$, $j$, and $k$ large enough, with $i \le j \le k$,
	\begin{align*}
		d(b, x_i) + d(x_i, x_j) + d(x_j, x_k)
   		 &< d(b, x_k) + 2 \epsilon \\
    	 &\le d(b, x_i) + d(x_i, x_k) + 2 \epsilon.
	\end{align*}
	The conclusion follows.
\end{proof}

Recall the following result.
If $x_n$ is an almost geodesic converging to a Busemann point $\xi$, then
\begin{align}
\label{eqn:optimal_along_ag}
\xi(x) = \lim_{n\to\infty} \big(d(x, x_n) + \xi(x_n)\big),
\qquad\text{for all $x \in X$}.
\end{align}

\begin{lem}
\label{lem:back_and_forth}
Let $x_n$ and $y_n$ be almost geodesics in a metric space $(X, d)$
converging to the same Busemann point $\xi$.
Then, there exists an almost geodesic $z_n$ that has infinitely many
points in common with $x_n$ and infinitely many points in common with $y_n$.
\end{lem}

\begin{proof}
Choose a sequence $\epsilon_i$ of positive real numbers such that
$\sum_{i=0}^\infty \epsilon_i$ is finite.
Define the sequence $z_i$ inductively in the following way.
Start with $z_0 := b$. Given $z_i$ with $i$ even,
use~(\ref{eqn:optimal_along_ag}) to define $z_{i+1} := x_n$,
where $n \ge i$ is large enough that
$\xi(z_i) > d(z_i, z_{i+1}) + \xi(z_{i+1}) - \epsilon_i$.
Given $z_i$ with $i$ odd, do the same but this time using the sequence $y_n$.

Observe that, from~(\ref{eqn:optimal_along_ag}),
$d(b, z_i) + \xi(z_i)$ converges to $\xi(b)= 0$, as $i$ tends to
infinity.

Since horofunctions are $1$-Lipschitz, $\xi(x) - \xi(y) \le d(x, y)$,
for all $x, y \in X$.
So, for all $m, n \in \NN$, with $m \leq n$,
\begin{align*}
d(z_m, z_n) 
    &\le \sum_{i=m}^{n-1} d(z_i, z_{i+1}) \\
    &< \xi(z_m) - \xi(z_n) + \sum_{i=m}^{n-1} \epsilon_i \\
    &\le \xi(z_m) + d(b, z_n) + \sum_{i=m}^{n-1} \epsilon_i.
\end{align*}
Adding $d(b, z_m)$ to both sides, we see that $z_i$ is an almost geodesic
because the error term $\sum_{i=m}^{n-1} \epsilon_i$
becomes arbitrarily small as $m$ and $n$ become large.
\end{proof}

We will now prove a convexity result for a pair of almost geodesics converging to the same Busemann point.

\begin{lem}[Convexity Lemma]
\label{lem:convexity lemma}
Let $x_n$ and $y_n$ be almost geodesics in a finite-dimensional normed space
$(X, ||\cdot||)$ converging to the same Busemann point $\xi$.
Let $\lambda_n$ be a sequence of coefficients in $[0, 1]$, and write
$m_n := (1 - \lambda_n) x_n + \lambda_n y_n$, for all $n\in\NN$.
Then $m_n$ converges to $\xi$, and it has an almost geodesic subsequence.
\end{lem}

\begin{proof}
Since the horofunction compactification is compact and metrizable,
to show that $m_n$ converges to $\xi$ it is enough to show that every
limit point $\eta$ of $m_n$ is equal to $\xi$.
By taking a subsequence if necessary, we may assume that $m_n$ converges
to $\eta$.

By Lemma~\ref{lem:back_and_forth}, there exists an almost geodesic sequence
$z_n$ that has infinitely many points in common with both $x_n$ and $y_n$.
Since almost geodesics always converge to a horofunction, $z_n$ has a limit,
which must necessarily be $\xi$.
By taking subsequences if necessary, we may assume that $z_n = x_n$
when $n$ is even, and $z_n = y_n$ when $n$ is odd.

Define the sequence
\begin{align}
w_n =
    \begin{cases}
    x_n, & \text{if $n$ is even}; \\
    m_n, & \text{if $n$ is odd}.
    \end{cases}
\end{align}
We will show that $w_n$ is an almost geodesic.

We first claim that, given any $\epsilon>0$, if $i,j,k\in \NN$
with $i < j < k$ are large enough and
such that $i$ and $k$ are even, and $j$ is odd, then
\begin{align}
      \label{eqn:no_short_cut}
      d(w_i, w_j) + d(w_j, w_k) < d(w_i, w_k) + \epsilon.
\end{align}
Here, $d(x, y) := || y - x||$ is the distance function associated to the
norm.

Indeed, note that the distance function $d(\cdot, \cdot)$ is convex
in each of its arguments. This implies that
\begin{align*}
      d(x_i, m_j)
      &\le (1 - \lambda_j) d(x_i, x_j) + \lambda_j d(x_i, y_j)
      \quad\text{and} \\
      d(m_j, x_k)
      &\le (1 - \lambda_j) d(x_j, x_k) + \lambda_j d(y_j, x_k).
\end{align*}
Adding, and applying Lemma~\ref{lem:reverse_triangle_inequality} to the
almost geodesics $x_n$ and $z_n$, we get
\begin{align*}
      d(x_i, m_j) + d(m_j, x_k)
      &< d(x_i, x_k) + \epsilon,
\end{align*}
for $i$, $j$, and $k$ large enough.
This establishes the claim.

Let $m$ and $n$ be natural numbers satisfying $m < n$.
There are four cases, depending on whether $m$ and $n$ are even or odd.
We consider only the case where both are odd; the other cases
are similar but less complicated.
Using the triangle inquality, the claim just established, and that $x_i$
is an almost geodesic, we have,
for any $\epsilon > 0$,
\begin{align*}
      d(b, w_m) &+ d(w_m, w_n)
	  \le d(b, w_{m-1}) + d(w_{m-1}, w_m)
	  + d(w_m, w_{m+1}) + d(w_{m+1}, w_n) \\
      &< d(b, w_{m-1}) + d(w_{m-1}, w_{m+1})
	  + d(w_{m+1}, w_{n+1}) - d(w_n, w_{n+1}) + 2 \epsilon \\
      &\le d(b, w_{n+1}) - d(w_n, w_{n+1}) + 4 \epsilon \\
      &\le d(b, w_n) + 4 \epsilon,
\end{align*}
if $m$ and $n$ are large enough.
The same inequality can be proved in the other cases.
We conclude that $w_n$ is an almost geodesic.

Observe that both $\xi$ and $\eta$ are limit points of $w_n$.
Since this sequence is an almost geodesic, it has a limit.
Hence, $\xi$ and $\eta$ are equal.
\end{proof}

\section{The compactification of flats in symmetric spaces} \label{sec:flat to space}

Throughout this section, we will assume that the $W$-invariant vector norm on the flat $\aa$ is such that every horofunction is a Busemann point. According to~\cite[Theorem~1.2]{WalII}, this is equivalent to asking the set of extreme sets of the dual unit ball to be closed. This is a very mild condition, satisfied notably by every polyhedral norm.

\medskip

We will give, for any such $G$-invariant Finsler metric on the symmetric space, an explicit homeomorphism between the intrinsic compactification of a flat and the closure of a flat in the horofunction compactification of the symmetric space. \\

The \emph{intrinsic compactification} of the flat $A\cdot p_0$ is the horofunction compactification of $A\cdot p_0$ within  the space of continuous functions on $A\cdot p_0$, i.e. $\overline{\psi(A\cdot p_0)}^{\widetilde{C}(A \cdot p_0)}$. Since the exponential map is a diffeomorphism $\exp: \aa \to A\cdot p_0$, the intrinsic compactification is homeomorphic to the horofunction compactification of the normed vector space $\aa$ with respect to the norm defined by the $W$-invariant convex ball $B$. In the cas of a polyhedral norm, this has been determined explicitely in \cite{JS}: 

\begin{thrm}[\cite{JS} Theorem 1.2.]\label{thrm:flatcompactification}
      Let $(V,\norm)$ be a normed vector space with polyhedral unit ball $B$. Then the horofunction compactification $\overline{V}^{hor}$ is homeomorphic to the dual convex polyhedron $B^\circ$. 
\end{thrm}

Let $d$ be the distance function associated to a $G$-invariant Finsler metric on the symmetric space $X = G/K$, and $\psi: X \to \widetilde{C}(X), \quad z \mapsto \psi_z$ the embedding defined in Subsection \ref{subsec:defi_of_horofunction} on page \pageref{subsec:defi_of_horofunction}. Let us state some basic observations.

\begin{lem}\label{K_invariance}
      The function  $\psi_{p_0}$ is $K$-invariant. Moreover for every $g \in G$, the function  $\psi_{g\cdot p_0}$ is $gKg^{-1}$-invariant. 
\end{lem}

\begin{proof}
      Fix $g \in G$ and $k \in K$. Then, for any $x \in X$, we have
      \begin{align*} 
	  \psi_{g\cdot p_0}((gkg^{-1}) \cdot x) &= d((gkg^{-1}) \cdot x,g \cdot p_0)-d(p_0,g \cdot p_0)\\
	  &=d(x,g \cdot p_0)-d(p_0,g \cdot p_0)=\psi_{g\cdot p_0}(x).
      \end{align*}
      So $\psi_{g\cdot p_0}$ is $gKg^{-1}$-invariant.
\end{proof}

\begin{lem}\label{K_equivariance}
      The map $\psi: X \to \widetilde{C}(X)$ is $K$-equivariant, that is, $\psi_{k\cdot z} (x) = k \cdot \psi_z(x)$, where the action on $\widetilde{C}(X)$ is given by $k\cdot f(x):= f(k^{-1} x)$. 
\end{lem}

\begin{proof}
      Fix $x,z \in X$ and $k \in K$. Then
      \begin{align*}
	  \psi_{k \cdot z}(x) &= d(x,k \cdot z) - d(p_0,k \cdot z) = d(k^{-1} \cdot x,z)-d(p_0,z)\\
	&=\psi_z(k^{-1} \cdot x)=k \cdot \psi_z(x). \qedhere
      \end{align*} 
\end{proof}

\begin{lem}
      Let $G = K\overline{A^+} K$ be a Cartan decompostion, and $X = K\overline{A^+} \cdot p_0$. Then 
      \[
	  \overline{\psi(X)}^{\widetilde{C}(X)} = \overline{\psi(K\overline{A^+} \cdot p_0)}^{\widetilde{C}(X)} = K \overline{\psi(\overline{A^+} \cdot p_0)}^{\widetilde{C}(X)}. 
      \]
      In particular, the horofunction compactification $\overline{\psi(X)}^{\widetilde{C}(X)}$ is determined by the horofunction compactification of the flat $F = A \cdot p_0$, or more precisely of a closed Weyl chamber $F^+ = \overline{A^+} \cdot p_0$. 
\end{lem}

\begin{proof}
      Since $\overline{\psi(\overline{A^+} \cdot p_0)}^{\widetilde{C}(X)}$ is a compact subspace of $\widetilde{C}(X)$ and $K$ is a compact subgroup of $G$, which acts continuously on $\widetilde{C}(X)$, we deduce that the space $K \overline{\psi(\overline{A^+} \cdot p_0)}^{\widetilde{C}(X)}$ is a compact subspace of $\widetilde{C}(X)$. Since it contains $\psi(K\overline{A^+} \cdot p_0)$, we conclude that $\overline{\psi(K\overline{A^+} \cdot p_0)}^{\widetilde{C}(X)} \subset K \overline{\psi(\overline{A^+} \cdot p_0)}^{\widetilde{C}(X)}$. As the converse inclusion is clear, we conclude that 
      \[
      		\overline{\psi(K\overline{A^+} \cdot p_0)}^{\widetilde{C}(X)} = K \overline{\psi(\overline{A^+} \cdot p_0)}^{\widetilde{C}(X)}. \qedhere
      \]
\end{proof}

In order to understand the horofunction compactification $\overline{\psi(\overline{A^+} \cdot p_0)}^{\widetilde{C}(X)}$ of a closed Weyl chamber, we will first compare it to  its closure in the intrinsic horofunction compactification in $\widetilde{C}(A \cdot p_0)$.

\subsection{The closure of a flat} 
The aim of this section is to compare the intrinsic compactification of $A\cdot p_0$ with the closure of the flat $A\cdot p_0$ in the horofunction compactification of $X$. To minimize confusion, we introduce the following notation: 

Let $d$ be the distance function of a $G$-invariant Finsler metric on $X = G/K$ and 
\begin{align} \label{Psi_X}
      \psi^X: X &\lora \widetilde{C}(X) \nonumber \\
      z &\longmapsto \psi^X_z \coloneqq d(\cdot,z) - d(p_0,z)
\end{align}
the embedding of $X$ into the space of continuous functions on $X$ vanishing at $p_0$. 

We denote by $d$ also the restriction of the distance function to the flat $F= A\cdot p_0 \subset X$ and let 
\begin{align} \label{Psi_F}
      \psi^F: F &\lora \widetilde{C}(F) \nonumber \\
      z &\longmapsto \psi^F_z \coloneqq d(\cdot,z) - d(p_0,z)
\end{align}
denote the embedding of $F$ into the space of continuous functions on $F$ vanishing at $p_0$. The closure of $\psi^F(F) \subset  \widetilde{C}(F)$ is the intrinsic compactification of $F$. We set $F^+ \coloneqq \overline{A^+} \cdot p_0$. 

\subsection{Groups associated with subsets of simple roots}

We want to associate to each horofunction on the flat a horofunction on the whole symmetric space $X$. We will start by setting up notation.

Let $\Delta$ be the set of positive roots. Given a subset $I \subset \Delta$ we denote by 

      \begin{tabularx}{\textwidth}{lX}
	  $W_I < W$ & the subgroup generated by the reflections in the hyperplanes $\ker(\alpha)$ for $\alpha \in I$,  \\
	  $\frak{a}_I= \bigcap_{\alpha \in I} \ker \alpha$,\\ 
	  $\frak{a}^I$ & the orthogonal complement of $\frak{a}_I$ in $\frak{a}$, \\
	  $A_I$, $A^I$  & the connected subgroups of $A$ with Lie algebras $\frak{a}_I$ and $\frak{a}^I$ respectively, \\
	   $M=C_K(A)$ & the centralizer of $A$ in $K$, \\
	  $G^I$  & the derived subgroup of the centralizer of $A_I$ in $G$, \\
	  $K^I= G^I \cap K$, & so that $K^IM$ is the centralizer of $A_I$ in $K$, \\
	  \multicolumn{2}{l}{$W^I= N_{K^I}(A^I)/Z_{K^I}(A^I)$  the Weyl group of $G^I$,} \\
	  $N$ & the connected subgroup with Lie algebra $\bigoplus_{\alpha \in \Sigma^+} \frak{g}_\alpha$, \\
	  $N_I$ & the connected subgroup of $N$ with Lie algebra $\bigoplus_{\alpha \in \Sigma^+ \bs \Sigma^I} \frak{g}_\alpha$, where $\Sigma^I$ is the root subsystem spanned by $I$.
      \end{tabularx}

\begin{defi}
Two subsets $I,J$ of $\Delta$ are said to be \emph{orthogonal} if, for every $\alpha \in I$ and $\beta \in J$, the roots $\alpha$ and $\beta$ are orthogonal. A subset $I \subset \Delta$ is called \emph{irreducible} if it is not a disjoint union of two proper orthogonal subsets.
\end{defi}

\begin{lem} \label{WI_invariance}
Fix a subset $I$ of $\Delta$, and consider a linear subspace $V$ of $\frak{a}^I$ which is invariant under the action of $W^I$. Then there exists $J \subset I$ such that $V=\frak{a}^J$, and $J$ and $I \bs J$ are orthogonal.
\end{lem}

\begin{proof}
Let $I = J_1 \sqcup J_2 \sqcup \dots \sqcup J_r$ be the decomposition of $I$ into irreducible subsets. The linear representation of $W^I$ on $\frak{a}^I$ decomposes as the direct sum of the irreducible representations $\frak{a}^I = \displaystyle \bigoplus_{j=1}^r \frak{a}^{J_j}$. Since $V$ is a $W^I$-invariant subspace, there exists $R \subset \{1,2,\dots,r\}$ such that $V=\displaystyle \bigoplus_{j \in R} \frak{a}^{J_j}$. As a consequence, we have $V=\frak{a}^J$, where $J=\displaystyle \bigsqcup_{j \in R} J_j$.
\end{proof}

\begin{lem} \label{lem:invariance subgroup} 
      Let $C$ be a non-discrete subset of $A^I$. Let $J' \subset I$ denote the smallest subset such that
      \begin{itemize}
	  \item[i)] $C \subset cA^{J'}$ forall $c \in C$,  and
	  \item[ii)] the roots in $J'$ and in $I \bs J'$ are orthogonal.
      \end{itemize} 
      Then the smallest closed subgroup of $W^I A$ containing all conjugates \linebreak 
	$\{cW^Ic^{-1}, c \in C\}$ is equal to $W^I A^{J'}$. 
\end{lem}

\begin{proof}
      In this proof, we will identify $A$ with its Lie algebra, and thus consider $A$ as a vector space. Up to conjugating, we can assume that the affine subspace of $A$ spanned by $C$ contains $0$.
      Let $\Gamma \subset W^I A$ denote the smallest closed subgroup containing all conjugates \mbox{$\{cW^Ic^{-1}, c \in C\}$.} Since $C$ is non-discrete, $\Gamma$ is not discrete and the linear part of $\Gamma$ is equal to $W^I$. So the identity component $\Gamma_0$ of $\Gamma$ is a vector subspace of $A^I$ containing $C$. Since $\Gamma_0$ is invariant under $W^I$, we deduce according to Lemma~\ref{WI_invariance} that $\Gamma_0=A^{J'}$, for some $J' \subset I$ such that $J'$ and $I \bs J'$ are orthogonal.
\end{proof}

\subsubsection{Generalized horocyclic decompositions} 

We will make use of the generalized Iwasawa decompositions of $G$, respectively the generalized horocyclic decompositions of $X$. 

\begin{lem} \label{lem:relative Iwasawa}
      For every $I \subset \Delta$ and $a^I \in A^I$, we have the following decomposition:
      \[
	  X=a^IK^I{a^I}^{-1}N_IA\cdot p_0,
      \] 
      where the $A$ component is unique up to the following condition: for every $a,a' \in A$, we have $a^IK^I{a^I}^{-1}N_Ia\cdot p_0 = a^IK^I{a^I}^{-1}N_Ia'\cdot p_0$ if and only if $(a^I)^{-1}a$ and $(a^I)^{-1}a'$ are conjugated by some element in $W^I$. The classical Iwasawa and horocyclic decompositions $G = NAK$ resp. $X= NA\cdot p_0$ correspond to $I = \emptyset$.
\end{lem}

\begin{proof}
      Up to translating by ${a^I}^{-1}$, we can assume for simplicity that $a^I=e$. According to \cite[Corollary 2.16]{GJT}, we have the following generalized horocyclic decomposition: $X=A_IN_IX^I$, where $X^I$ is the relative symmetric space $X^I=G^I / K^I$ identified as the orbit $X^I = G^I \cdot p_0$ of $p_0$ in $X$. Furthermore, in this decomposition  $X=A_IN_IX^I = A_I N_I G^I \cdot p_0$, the components in $A_I$, $N_I$ and $X^I \simeq G^I \cdot p_0$ are unique.

The group $K^I$ is a maximal compact subgroup of the semisimple group $G^I$, and $A^I$ is a Cartan subgroup of $G^I$, so we can consider the Cartan decomposition of $G^I$ as $G^I = K^I A^I K^I$, where the component in $A^I$ is unique up to conjugation by some element in $W^I$. 

Fix some point $p \in X$. According to the two previous decompositions, we have $p=b_Iu_Ik^Ib^I \cdot p_0$, where $b_I \in A_I$, $u_I \in N_I$, $k^I \in K^I$ and $b^I \in A^I$, and furthermore $b_I$ and $u_I$ are unique and $b^I$ is unique up to conjugation by some element in $W^I$. Since $A_I$ commutes with $K^I$, we also have $p=(b_Iu_Ib_I^{-1})k^Ib_Ib^I \cdot p_0$. Furthermore, since $A^I$ and $K^IM$ normalize $N_I$, we have $(b_Iu_Ib_I^{-1})k^I \in K^IN_I$.

As a consequence, $p \in K^IN_I b_Ib^I \cdot p_0$, where $b_Ib^I \in A$ is unique up to conjugation by some element in $W^I$ (notice that $W^I$ commutes with $b_I \in A_I$).
\end{proof}

\subsubsection{Types of sequences and horofunctions}

We have seen in Lemma~\ref{K_invariance} that each function $\psi_{g \cdot p_0}$ is invariant under the conjuagte $gKg^{-1}$ of the maximal compact subgroup $K$. In order to study the invariance properties of horofunctions, we will use the study of limits of conjugates of $K$ (see~\cite[Chapter~IX]{GJT}). In order to describe such limits, we need to introduce the notion of type of a diverging sequence of elements in $A$. Roughly speaking, the type of a sequence encodes the roots "along which" the sequence goes to infinity.

\begin{defi} 
      A sequence $(a_n)_{n \in \NN}$ in $\overline{A^+}$ is said to be of \emph{type $(I,a^I)$}, where $I$ is a proper subset of $\Delta$ and $a^I \in A^I$, if
      \begin{itemize}
	  \item[i)] for $\alpha \in I$, $\lim_{n \rightarrow \infty} \alpha(\log a_n)$ exists and is equal to $\alpha(\log a^I)$,
	  \item[ii)] for $\alpha \in \Delta \bs I$ there holds $\alpha(\log a_n) \rightarrow +\infty$.
      \end{itemize}
\end{defi}

The main result on limits of conjugates of $K$ is the following.

\begin{prop}{\cite[Proposition~9.14]{GJT}} \label{limits of K}
Let $(a_n)_{n \in \NN}$ be a sequence in $\overline{A^+}$ of type $(I,a^I)$. In the space of closed subgroups of $G$, endowed with the Chabauty topology, the sequence $(a_nK{a_n}^{-1})_{n \in \NN}$ converges to $a^IK^IM(a^I)^{-1}N_I$.
\end{prop}

Recall that we gave a short definition of the Chabauty topology after Theorem~\ref{thm:walsh}.

\begin{rem}
Since the groups $a^IK^IM{a^I}^{-1}N_I$ arise as limits of the maximal compact subgroups under conjugations by sequences of type $I$ in $A$, 
the (generalized) Iwasawa decompositions can thus be seen as  limits of the Cartan decomposition.
\end{rem}

We will now use this result to deduce some invariance for horofunctions.

\begin{lem} \label{lem:limit invariance} 
      Let $(a_n)_{n \in \NN}$ be a sequence in $\overline{A^+}$ of type $(I,a^I)$ such that $(\psi^X_{a_n \cdot p_0})_{n \in \NN}$ converges to $\xi$. Then $\xi$ is $a^IK^IM(a^I)^{-1}N_I$-invariant.
\end{lem}

\begin{proof}
      For each $n \in \NN$, the function $\psi^X_{a_n \cdot p_0}$ is invariant under $a_nKa_n^{-1}$. Since the sequence $(a_nKa_n^{-1})_{n \in \NN}$ converges to $a^IK^IM(a^I)^{-1}N_I$ in the Chabauty topology (see~Proposition~\ref{limits of K}), for every $g \in a^IK^IM(a^I)^{-1}N_I$ there exists a sequence $(k_n)_{n \in \NN}$ in $K$ such that the sequence $(a_nk_na_n^{-1})_{n \in \NN}$ converges to $g$. Therefore, for every $p \in X$ we have
      \begin{eqnarray*} 
	  \xi(g \cdot p)-\xi(p) &=& \lim_{n \rightarrow +\infty} d(a_n \cdot p_0,g \cdot p)-d(a_n \cdot p_0,p) \\
	  &=& \lim_{n \rightarrow +\infty} d(a_n \cdot p_0,a_nk_na_n^{-1} \cdot p)-d(a_n \cdot p_0,p) = 0. 
      \end{eqnarray*}
      As a consequence, $\xi$ is invariant under $a^IK^IM(a^I)^{-1}N_I$.
\end{proof}

\begin{defi} 
      A horofunction $\eta \in \partial \overline{\psi^F(F^+)}^{\widetilde{C}(F)}$ is said to be of \emph{type $(I,a^I)$}, where $I$ is a proper subset of $\Delta$ and $a^I \in A^I$, if there exists an almost geodesic sequence $(a_n)_{n \in \NN}$ in $A$ of type $(I,a^I)$ such that the sequence $(\psi^F_{a_n \cdot p_0})_{n \in \NN}$ converges to $\eta$ in $\widetilde{C}(F)$. Note that, since we assumed that every horofunction is a Busemann point, a horofunction may have several types, but has at least one type.
\end{defi}

\begin{lem} \label{lem:intersectiontype}
      Let $\eta \in \partial \overline{\psi^F(F^+)}^{\widetilde{C}(F)}$ be a horofunction which has two types $(I,a^I)$ and $(J,b^J)$, with $I, J \subset \Delta$ and $a^I \in A^I, b^J, \in A^J$. Then $\eta$ also has type $(I \cap J, c^{I \cap J})$ for some $c^{I \cap J} \in A^{I \cap J}$.
\end{lem}

\begin{proof}
      Let $(a_n)_{n \in \NN}$ and $(b_n)_{n \in \NN}$ be two almost geodesic sequences in $A$ of different types $(I,a^I)$ and $(J,b^J)$ respectively, such that the sequences $(\psi^F_{a_n \cdot p_0})_{n \in \NN}$ and $(\psi^F_{b_n \cdot p_0})_{n \in \NN}$ both converge to $\eta$. For every $n \in \NN$, we define \linebreak
      $c_n = \exp(\frac{1}{2}\log(a_n)+\frac{1}{2}\log(b_n))$. The sequence $(c_n)_{n \in \NN}$ has type $(I \cap J, c^{I \cap J})$, where $c^{I \cap J} \in \left(\exp(\frac{1}{2}\log(a^I)+\frac{1}{2}\log(b^J)) A_{I \cap J}\right) \cap A^{I \cap J}$. According to the Convexity Lemma (Lemma \ref{lem:convexity lemma}), the sequence $(\psi^A_{c_n \cdot p_0})_{n \in \NN}$ also converges to $\eta$. As a consequence, $\eta$ has type $(I \cap J, c^{I \cap J})$.
\end{proof}

\begin{lem} \label{lem:eta smaller type}
      Let $\eta \in \partial \overline{\psi^F(F^+)}^{\widetilde{C}(F)}$ be a horofunction of type $(I,a^I)$, where $I \subsetneq \Delta$ and $a^I \in A^I$. If $\eta$ is invariant under $A^{J'}$ with $J' \subset I$, then $\eta$ has type $(I \bs J',c^{I \bs J'})$ for some $c^{I \bs J'} \in A^{I \bs J'}$.
\end{lem}

\begin{proof}
      Fix $c \in {A^{J'}}^+$. For each $k \in \NN$, the sequence $(\psi^F_{c^k a_n \cdot p_0})_{n \in \NN}$ converges to $c^k \cdot \eta=\eta$, since $\eta$ is invariant under $A^{J'}$. As a consequence, there exists $n_k \in \NN$ such that, for every $n \geq n_k$, and for every $a \in A$ such that $d(p_0,a \cdot p_0) \leq k$, we have 
      \[
	  |d(a \cdot p_0,c^k a_n \cdot p_0) - d(a \cdot p_0,a_n \cdot p_0)| \leq \frac{1}{k+1}.
      \]
      We can furthermore assume that the sequence $(n_k)_{k \in \NN}$ is increasing. Fix $a \in A$. For every $k \geq d(p_0,a \cdot p_0)$, we have $|d(a \cdot p_0,c^k a_{n_k} \cdot p_0)-d(a \cdot p_0,a_{n_k} \cdot p_0)|$ $\leq \frac{1}{k+1}$ and $|d(p_0,c^k a_{n_k} \cdot p_0)-d(p_0,a_{n_k} \cdot p_0)| \leq \frac{1}{k+1}$. Therefore, we have 
      \begin{align*}
	  &\lim_{k \rightarrow +\infty} d(a \cdot p_0, c^k a_{n_k} \cdot p_0)-d(p_0,c^k a_{n_k} \cdot p_0)\\
		= &\lim_{k \rightarrow +\infty} d(a \cdot p_0,a_{n_k} \cdot p_0)-d(p_0,a_{n_k} \cdot p_0) \\
	  = & \ \eta(a \cdot p_0)-\eta(p_0).
      \end{align*}
      As a consequence, the sequence $(\psi^F_{c^k a_{n_k} \cdot p_0})_{k \in \NN}$ converges to $\eta$.
      
      To conclude, observe that the sequence $(c^k a_{n_k})_{k \in \NN}$ has type $(I \bs J',c^{I \bs J'})$, for some $c^{I \bs J'} \in A^{I \bs J'}$.
\end{proof}

\subsection{The intrinsic compactification versus the closure of a flat}
In this section we define an explicit map from the intrinsic compactification of the flat $F$ into the horofunction compactification of $X$. For this we use the invariance shown in Lemma \ref{lem:limit invariance} and the generalized horocyclic decomposition $X=a^IK^I{a^I}^{-1}N_IA\cdot p_0$ in Lemma~\ref{lem:relative Iwasawa}.

\begin{thrm} \label{thm:embedding flat in X} 
      The following map
      \begin{eqnarray*}
	  \phi : \overline{\psi^F(F^+)}^{\widetilde{C}(F)}& \longrightarrow & \overline{\psi^X(X)}^{\widetilde{C}(X)} \\
	  \psi^F_z, \mbox{ where }z \in F^+  & \longmapsto & \psi^X_z \\
	  \eta \mbox{ of type } (I,a^I) & \longmapsto &  \left(a^Ik^I{a^I}^{-1}u_Ia \cdot p_0 \in X \mapsto \eta(a \cdot p_0)\right)
      \end{eqnarray*}
      is a well-defined, continuous embedding.
\end{thrm}

\begin{proof} 
      The fact that $\phi$ is well-defined will be proved in Section~\ref{sec:well-definedness}, and the fact that $\phi$ is continuous will be proved in Section~\ref{sec:continuity}. Since the restiction to $F^+$ is a left inverse to $\phi$, we deduce that $\phi$ is injective. Since $\overline{\psi^F(F^+)}^{\widetilde{C}(F)}$ is compact, $\phi$ is then an embedding. 
\end{proof}

\subsubsection{Well-definedness} \label{sec:well-definedness}

We want to prove that the map $\phi$ in Theorem~\ref{thm:embedding flat in X} is well-defined.

\medskip

Consider first a horofunction $\eta \in \partial \overline{\psi^F(F^+)}^{\widetilde{C}(F)}$ which has some type $(I,a^I)$, and consider two decompositions $a^Ik^I{a^I}^{-1}u_Ia \cdot p_0 = a^Ik'^I{a^I}^{-1}u'_Ia' \cdot p_0$ of the same point in $X$. According to Lemma~\ref{lem:relative Iwasawa}, there exists $w \in W^I$ such that $(a^I)^{-1}a'=w (a^I)^{-1}a w^{-1}$. According to Lemma~\ref{lem:limit invariance}, $\eta$ is invariant under $a^IW^I(a^I)^{-1}$, so 
\[
\eta(a' \cdot p_0)=\eta(a^I w (a^I)^{-1}a w^{-1} \cdot p_0)=\eta((a^I w (a^I)^{-1})a \cdot p_0)=\eta(a \cdot p_0).
\]
 This means that the formula defining $\phi$ does not depend on the choice of the $A$ component in the decomposition $X=a^IK^I{a^I}^{-1}N_IA\cdot p_0$.

\medskip

Consider now a horofunction $\eta \in \partial  \overline{\psi^F(F^+)}^{\widetilde{C}(F)}$, which has two types $(I,a^I)$ and $(J,b^J)$. We will prove that the two formulas defining $\phi(\eta)$, for each type, agree. Let the notations be as in Lemma \ref{lem:intersectiontype}. Up to passing to a subsequence, we may assume that the sequences $(\psi^X_{a_n \cdot p_0})_{n \in \NN}$ and $(\psi^X_{b_n \cdot p_0})_{n \in \NN}$ converge to $\xi$ and $\xi'$ respectively. 
We precisely need to prove that $\xi=\xi'$, which will be done by induction on $|I|+|J|$. By Lemma \ref{lem:intersectiontype} assume from now on that $J \subset I$.

\medskip

Assume first that $|I|+|J|=0$, so $I=J=\emptyset$. According to Lemma~\ref{lem:limit invariance}, $\xi$ and $\xi'$ are both $N$-invariant, so for every $p = ua \cdot p_0 \in X=NA \cdot p_0$, we have $\xi(p)=\eta(a \cdot p_0)=\xi'(p)$. So $\xi=\xi'$.

\medskip

By induction, fix $m \in \NN$ and assume that if $|I| + |J| \leq m$, then $\xi=\xi'$. Consider now $I,J$ such that $|I|+|J|=m+1$. We will distinguish the two cases $J=I$ and $J \subsetneq I$. \\

\paragraph*{\underline{The case $J=I$}}

Assume that $J=I$. We will first show that $\eta$ has extra invariance and then define a subset $J' \subset I$ to show that $\eta$ has also a type smaller than $I$. The result will then follow by two inductions. 
      
\begin{lem} \label{lem:invariance when J is I} 
      Assume that $J=I$. Then there exists $J' \subset I$ such that :
      \begin{itemize}
	  \item[i)] $a^I \in b^IA^{J'}$,
	  \item[ii)] the roots in $J'$ and $I \bs J'$ are orthogonal, and
	  \item[iii)] $\eta$ is $W^IA^{J'}$-invariant.
      \end{itemize}
\end{lem}

\begin{proof}
      For simplicity, up to translating by $(a^{I})^{-1}$, we may assume that $a^I=e$.
  
      Fix $\lambda \in [0,1]$. For each $n \in \NN$, let $c_n=\exp((1-\lambda)\log a_n + \lambda \log b_n) \in A$. According to the Convexity Lemma \ref{lem:convexity lemma}, the sequence $(\psi^F_{c_n \cdot p_0})_{n \in \NN}$ converges to $\eta$.
      The sequence $(c_n)_{n \in \NN}$ is of type $(I, (b^I)^\lambda)$, where $(b^I)^\lambda$ denotes $\exp(\lambda \log b^I)$. Since the sequence $(\psi^F_{c_n \cdot p_0})_{n \in \NN}$ converges to $\eta$, by Lemma~\ref{lem:limit invariance} we deduce that $\eta$ is $(b^I)^\lambda W^{I}((b^I)^\lambda)^{-1}$-invariant, for every $\lambda \in [0,1]$.

      According to Lemma~\ref{lem:invariance subgroup}, we deduce that $\eta$ is invariant under $W^IA^{J'}$, where $J' \subset I$ is the smallest subset such that $b^I \in A^{J'}$ and such that the roots in $J'$ and in $I \bs J'$ are orthogonal.
\end{proof}

According to Lemma \ref{lem:eta smaller type}, we deduce that $\eta$ has also type $(I \bs J',c^{I \bs J'})$, for some $c^{I \bs J'} \in A^{I \bs J'}$. Let $(c_n)_{n \in \NN}$ denote a sequence of type $(I \bs J',c^{I \bs J'})$ such that the sequence $(\psi^F_{c_n \cdot p_0})_{n \in \NN}$ converges to $\eta$. Up to passing to a subsequence, assume that the sequence $(\psi^X_{c_n \cdot p_0})_{n \in \NN}$ converges to some $\xi''$.

Since $a^I \in b^IA^{J'}$ and $a^I \neq b^I$, we know that $J' \neq \emptyset$. Therefore we have $|I|+|I \bs J'| < |I|+|I|$ so $|I|+|I \bs J'| \leq m$. By induction applied to the sequences $(a_n)_{n \in \NN}$ and $(c_n)_{n \in \NN}$, we deduce that $\xi=\xi''$. By induction applied to the sequences $(b_n)_{n \in \NN}$ and $(c_n)_{n \in \NN}$, we deduce that $\xi'=\xi''$. In conclusion, we have $\xi=\xi'$. This concludes the induction, and finishes the proof that $\xi=\xi'$ in the case where $J = I$.\\

\paragraph*{\underline{The case $J \subsetneq I$}}

Assume that $J \subsetneq I$. Similarly to the case before, we will first show an extra invariance of $\eta$ and that it has a smaller type with respect to a new subset $J' \subset I$. To conclude the result by induction, we have to distinguish again two cases depending on whether $I \setminus J' = J$ or not. 

\begin{lem} \label{lem:invariance when J not I} There exists $J' \subset I$ such that :
      \begin{itemize}
	  \item[i)] $J \cup J' =I$,
	  \item[ii)] the roots in $J'$ and $I \bs J'$ are orthogonal, and
	  \item[iii)] $\eta$ is $W^IA^{J'}$-invariant.
      \end{itemize}
\end{lem}

\begin{proof}
Let $a_n, b_n$ be the sequences of type $I$ and $J$ converging to $\eta$. 
      For simplicity, up to translating by $(a^{I})^{-1}$, we may assume that $a^I=e$.
	Up to passing to a subsequence, let us partition $I \bs J$ into $I \bs J= I_1 \sqcup I_2 \sqcup \dots \sqcup I_p$ such that:
      \begin{itemize}
	  \item $\forall 1 \leq i \leq p, \forall \alpha,\beta \in I_i, \lim_{n \rightarrow +\infty} \frac{\alpha(\log b_n)}{\beta(\log b_n)} \in (0,+\infty)$,
	  \item $\forall 1 \leq i < j \leq p, \forall \alpha \in I_i, \forall \beta \in I_j, \lim_{n \rightarrow +\infty} \frac{\alpha(\log b_n)}{\beta(\log b_n)} =0$.
      \end{itemize}

      \medskip

      Fix $1 \leq i \leq p$, and for some $\alpha \in I_i$ define $t_n \coloneqq \frac{1}{\alpha(\log b_n)}$ such that $t_n \longrightarrow 0$ as $n \rightarrow +\infty$. Fix $\lambda >0$. For each $n \in \NN$, let $c_n=\exp((1-\lambda t_n)\log a_n + \lambda t_n \log b_n) \in A$. According to Lemma \ref{lem:convexity lemma}, the sequence $(\psi^F_{c_n \cdot p_0})_{n \in \NN}$ converges to $\eta$. Let us define 
      \[
	  c^{I_i} \coloneqq \lim_{n \ra +\infty} \left( \pi^{I_i}(b_n) \right)^{t_n} \in A^{I_i},
      \]
      where $\pi^{I_i}(b_n)$ denotes the orthogonal projection of $b_n$ onto $A^{I_i}$. Note that this sequence converges: for any $\beta \in I_i$, we have
\[\beta \left(\log \left( \pi^{I_i}(b_n) \right)^{t_n}\right) = t_n \beta(\log b_n) = \frac{\beta(\log b_n)}{\alpha(\log b_n)},\]
so $\lim_{n \rightarrow +\infty} \beta \left(\log \left( \pi^{I_i}(b_n) \right)^{t_n}\right) \in (0,+\infty)$. On the other hand, for any $\beta \in \Delta \bs I_i$, we have \mbox{$\beta \left(\log \left( \pi^{I_i}(b_n) \right)^{t_n}\right)=0$}, so the limit $c^{I_i} \in A^{I_i}$ exists.

Furthermore, we have $c^{I_i} \in (A^{I_i})^{+}$. Let 
      \[
	  J_i \coloneqq J \sqcup I_1 \sqcup \dots \sqcup I_i.
      \]
      For every $\alpha \in \Delta \bs J_i$, we have 
      \[
	  \alpha(\log c_n)=(1-\lambda t_n)\alpha(\log a_n) + \lambda t_n \alpha(\log b_n) \longrightarrow +\infty.
      \]
      For every $\alpha \in J \cup I_1 \cup \dots \cup I_{i-1}$, we have 
      \[
         \alpha(\log c_n) =(1-\lambda t_n)\alpha(\log a_n) + \lambda t_n \alpha(\log b_n) \longrightarrow \alpha(\log a^I) = 0.\]
      For every $\alpha \in I_i$, we have
      \begin{align*}
	  \alpha(\log c_n) &=(1-\lambda t_n)\alpha(\log a_n) + \lambda t_n \alpha(\log b_n) \\
	  &\longrightarrow \alpha(\log a^I) + \lambda \alpha(\log c^{I_i}) = \lambda \alpha(\log c^{I_i}).
      \end{align*}
      As a consequence, the sequence $(c_n)_{n \in \NN}$ is of type $(J_i, (c^{I_i})^\lambda)$, where $(c^{I_i})^\lambda$ denotes $\exp(\lambda \log c^{I_i})$. Since the sequence $(\psi^F_{c_n \cdot p_0})_{n \in \NN}$ converges to $\eta$, by Lemma~\ref{lem:limit invariance} we deduce that $\eta$ is $(c^{I_i})^\lambda W^{I_i}((c^{I_i})^\lambda)^{-1}$-invariant.

      As $c^{I_i} \in (A^{I_i})^{+}$, we deduce by Lemma \ref{lem:invariance subgroup} that $\eta$ is invariant under $A^{I_i}$. Because this is true for every $1 \leq i \leq p$, we conclude that $\eta$ is invariant under $A^{I \bs J}$.

      Since the sequence $(a_n)_{n \in \NN}$ is of type $(I,e)$, and the sequence $(\psi^F_{a_n \cdot p_0})_{n \in \NN}$ converges to $\eta$, we know by Lemma \ref{lem:limit invariance}  that $\eta$ is $W^I$-invariant. In conclusion, $\eta$ is invariant under $W^I$ and $A^{I \bs J}$. The smallest closed subgroup of $W^IA^I$ containing both $W^I$ and $A^{I \bs J}$ is $W^IA^{J'}$, where $J' \subset I$ is the smallest subset containing $I \bs J$ such that the roots in $J'$ and in $I \bs J'$ are orthogonal. Therefore $\eta$ is invariant under $W^IA^{J'}$.
\end{proof}

Since $\eta$ is invariant under $A^{J'}$,  we deduce according to Lemma \ref{lem:eta smaller type} that $\eta$ has type $(I \bs J',c^{I \bs J'})$, for some $c^{I \bs J'} \in A^{I \bs J'}$.

\medskip

If $I \bs J' \subsetneq J$, then $|I|+|I \bs J'|<|I|+|J|$ and $|J|+|I \bs J'|<|I|+|J|$, so by applying the induction twice, we know that $\xi=\xi'$.

\medskip

So we are left with the case $I \bs J' = J$. In this case $J$ and $I \bs J$ are orthogonal.

\begin{lem} \label{lem:case J orthogonal in I} 
      If $\eta$ is $A^{I \bs J}$-invariant and $J$ and $I \bs J$ are orthogonal, then $\xi=\xi'$. 
\end{lem}

\begin{proof}
      As $J$ and $J'=I \bs J$ are orthogonal, we have the orthogonal decomposition $A^I=A^JA^{J'}$. Let us decompose $a^I=a^J a^{J'} \in A^JA^{J'}$. Up to translating by $(b^Ja^{J'})^{-1}$, we can assume that $b^J=e$ and $a^I = a^{J'} \in A^{J'}$.

      \medskip

      As $\Sigma^J$ and $\Sigma^{J'}$ are orthogonal, we have the decomposition $K^I=K^JK^{J'}$, with $K^J$ and $K^{J'}$ commuting. Furthermore $K^J$ and $A_J$ are commuting. Since $A^{J'} \subset A_J$, we deduce that $a^{J'}$ commutes with $K^J$. In particular, 
      \[
	  a^{J'}K^I(a^{J'})^{-1}=K^Ja^{J'}K^{J'}(a^{J'})^{-1}.
      \]
      Fix any point $p \in X$, we will show that $\xi'(p)=\xi(p)$. Consider the decomposition $X = a^{J'}K^I(a^{J'})^{-1}N_IA \cdot p_0 = K^JN_Ia^{J'}K^{J'}(a^{J'})^{-1}A \cdot p_0$, and write 
      \[
	  p=k^Ju_Ia^{J'}k^{J'}(a^{J'})^{-1}c \cdot p_0 \in X,
      \]
      where $k^J \in K^J$, $u_I \in N_I$, $k^{J'} \in K^{J'}$ and $c \in A$. According to Lemma \ref{lem:limit invariance}, $\xi'$ is invariant under $K^JMN_J$. Since $N_I \subset N_J$,  we deduce that 
      \[
	  \xi'(p)=\xi'(a^{J'}k^{J'}(a^{J'})^{-1}c \cdot p_0). 
      \]
      In the decomposition $A=A_{J'}A^{J'}$, let us write $c=c_{J'}c^{J'}$. So 
      \[
	  a^{J'}k^{J'}(a^{J'})^{-1}c = c_{J'}a^{J'}k^{J'}(a^{J'})^{-1}c^{J'} \in c_{J'}G^{J'}.
      \]
      According to the Iwasawa decomposition $G^{J'}=N^{J'}A^{J'}K^{J'}$, there exists $u^{J'} \in N^{J'}$ and $d^{J'} \in A^{J'}$ such that $u^{J'}a^{J'}k^{J'}(a^{J'})^{-1}c^{J'} \in d^{J'}K^{J'}$. As a consequence, 
      \[
	  u^{J'}\cdot (a^{J'}k^{J'}(a^{J'})^{-1}c \cdot p_0)=c_{J'}d^{J'} \cdot p_0.
      \]
      We claim that 
      \[
	  \xi'(p)=\xi'(c_{J'}d^{J'} \cdot p_0).
      \]
      Showing this is equivalent to showing that $\xi'(a^{J'}k^{J'}(a^{J'})^{-1}c \cdot p_0) =$ \linebreak $\xi'(u^{J'}a^{J'}k^{J'}(a^{J'})^{-1}c \cdot p_0)$. Since the sequence $(b_n K b_n^{-1})_{n \in \NN}$ converges to $K^JMN_J$ in the Chabauty topology (see~Proposition~\ref{limits of K}), and as $u^{J'} \in N^{J'} \subset N_J$, there exists a sequence $(k_n)_{n \in \NN}$ such that the sequence $(b_nk_n{b_n}^{-1})_{n \in \NN}$ converge to $u^{J'}$. As a consequence,
      
      \begin{eqnarray*}
	  \xi'(u^{J'}a^{J'}k^{J'}(a^{J'})^{-1}c \cdot p_0)-\xi'(a^{J'}k^{J'}(a^{J'})^{-1}c \cdot p_0)&=& \\
	  \lim_{n \ra +\infty} d(b_n \cdot p_0,u^{J'}a^{J'}k^{J'}(a^{J'})^{-1}c \cdot p_0) - d(b_n \cdot p_0,a^{J'}k^{J'}(a^{J'})^{-1}c \cdot p_0)&=& \\
	  \lim_{n \ra +\infty} d(b_n \cdot p_0,b_nk_nb_n^{-1}a^{J'}k^{J'}(a^{J'})^{-1}c \cdot p_0) - d(b_n \cdot p_0,a^{J'}k^{J'}(a^{J'})^{-1}c \cdot p_0)&=&0.
      \end{eqnarray*}
      Hence $\xi'(u^{J'}a^{J'}k^{J'}(a^{J'})^{-1}c \cdot p_0) = \xi'(a^{J'}k^{J'}(a^{J'})^{-1}c \cdot p_0)$, so 
      \[
	  \xi'(p)=\xi'(c_{J'}d^{J'} \cdot p_0).
      \]
      By assumption, $\eta$ is invariant under $A^{I \bs J}=A^{J'}$. As a consequence, we have 
      \[
	  \xi'(p)=\xi'(c_{J'}d^{J'} \cdot p_0) = \eta(c_{J'}d^{J'} \cdot p_0)=\eta(c_{J'} \cdot p_0).
      \]
      On the other hand, according to Lemma \ref{lem:limit invariance}, we have 
      \[
	  \xi(p)=\xi(k^Ju_Ia^{J'}k^{J'}(a^{J'})^{-1}c \cdot p_0)=\xi(c \cdot p_0)=\eta(c \cdot p_0).
      \]
      Since $c=c_{J'}c^{J'}$ and $\eta$ is invariant under $A^{J'}$, we conclude that \linebreak $\xi(p)=\eta(c_{J'} \cdot p_0)$. Therefore, $\xi'(p)=\xi(p)$. So $\xi=\xi'$.
\end{proof}

This conludes the proof by induction that $\xi=\xi'$. So we have proved that the map $\phi$ in Theorem \ref{thm:embedding flat in X} is well-defined.

\subsubsection{Continuity} \label{sec:continuity}

We want to prove that the map $\phi$ in Theorem \ref{thm:embedding flat in X} is continuous. It is clear that $\phi$ is continuous on the interior $\psi^F(F^+)$. Fix $\eta \in \partial\overline{\psi^F(F^+)}^{\widetilde{C}(F)}$, we will show that $\phi$ is continuous at $\eta$.

\begin{lem}
      Let $(a_n)_{n \in \NN}$ be an almost geodesic sequence in $\overline{A^+}$ such that the sequence $(\psi^F_{a_n \cdot p_0})_{n \in \NN}$ converges to $\eta$. Then $(\psi^X_{a_n \cdot p_0})_{n \in \NN}$  converges to $\phi(\eta)$.
\end{lem}

\begin{proof}
      Up to passing to a subsequence, we may assume that the sequence $(a_n)_{n \in \NN}$ has some type $(I,a^I)$ and that the sequence $(\psi^X_{a_n \cdot p_0})_{n \in \NN}$  converges to some $\xi$. According to Lemma \ref{lem:limit invariance}, $\xi$ is invariant under $a^IK^IM(a^I)^{-1}N_I$, so for every $p=a^Ik^I(a^I)^{-1}u_Ia \cdot p_0 \in X=a^IK^I(a^I)^{-1}N_IA \cdot p_0$, we have $\xi(p)=\xi(a \cdot p_0)=\eta(a \cdot p_0)$.

      Furthermore, since $\phi$ is well-defined and $\eta$ has type $(I,a^I)$, we can use this type in the definition of $\phi(\eta)$, and thus $\phi(\eta)(p)=\eta(a \cdot p_0)=\xi(p)$. In conclusion, $\xi=\phi(\eta)$, so $(\psi^X_{a_n \cdot p_0})_{n \in \NN}$  converges to $\phi(\eta)$ in $C(X)$.
\end{proof}

\begin{lem}
      Let $(\eta_n)_{n \in \NN}$ be a sequence in $\partial\overline{\psi^F(F^+)}^{\widetilde{C}(F)}$ converging to $\eta$ in $\widetilde{C}(F)$. Then $(\phi(\eta_n))_{n \in \NN}$ converges to $\phi(\eta)$ in $\widetilde{C}(X)$.
\end{lem}

\begin{proof}
      Up to passing to a subsequence, we may assume that the sequence $(\phi(\eta_n))_{n \in \NN}$ converges to some horofunction $\xi$ in $\widetilde{C}(X)$. Up to passing again to a subsequence, we may assume that there exists $I \subsetneq \Delta$ such that for each $n \in \NN$, $\eta_n$ is of type $(I,a_n^I)$ for some $a_n^I \in A^I$. For each $n \in \NN$, consider a sequence $(a_{n,m})_{m \in \NN}$ of type $(I,a_n^I)$ converging to $\eta_n$. Up to passing to a subsequence, we may assume that the sequence $(a_n^I)_{n \in \NN}$ is of type $(J,a^J)$ for some $J \subset I$ and some $a^J \in A^J$. For each $n \in \NN$, one can find some $m_n \in \NN$ such that the sequence $(a_{n,m_n})_{n \in \NN}$ is of type $(J,a^J)$ and converges to $\eta$.
 
      Fix 
      \[
           p=a^Jk^J{a^J}^{-1}u_Jc \cdot p_0 \in X=a^JK^J{a^J}^{-1}N_JA \cdot p_0.
      \]
 Since the sequence $(a_n^IK^IM{a_n^I}^{-1}N_I)_{n \in \NN}$ converges to $a^JK^JM{a^J}^{-1}N_J$ in the Chabauty topology (see~Proposition~\ref{limits of K}), there exist sequences $(k_n^I)_{n \in \NN}$ in $K^IM$ and $(u_{n,I})_{n \in \NN}$ in $N_I$ such that the sequence $(a_n^Ik_n^I{a_n^I}^{-1}u_{n,I})_{n \in \NN}$ converges to $a^Jk^J{a^J}^{-1}u_J$. Hence 
 \begin{align*}
      \xi(p) &= \lim_{n \rightarrow +\infty} \phi(\eta_n)(a_n^Ik_n^I{a_n^I}^{-1}u_{n,I} c \cdot p_0) \\
      &= \lim_{n \rightarrow +\infty} \eta_n(c \cdot p_0) \\
      &= \eta(c \cdot p_0) \\
      &= \phi(\eta)(a^Jk^J{a^J}^{-1}u_Jc \cdot p_0) = \phi(\eta)(p).
 \end{align*}

      As a consequence, we have $\xi=\phi(\eta)$, so the sequence $(\phi(\eta_n))_{n \in \NN}$ converges to $\phi(\eta)$.
\end{proof}

So we have proved that the map $\phi$ in Theorem~\ref{thm:embedding flat in X} is continuous. This concludes the proof of Theorem~\ref{thm:embedding flat in X}.

\section{Realizing classical compactifications of symmetric spaces} \label{sec:satake}
In this section we prove that all Satake and generalized Satake compactifications can be realized as horofunction compactifications of polyhedral $G$-invariant Finsler metrics on $X$.

\subsection{Generalized Satake compactifications}
We first recall the construction of generalized Satake compactifications. See also \cite{GuichardKasselWienhard}. Let $X = G/K$ be a symmetric space of non-compact type.

Consider the space
\[
 \Pn \coloneqq \rquotient{\PSL(n,\CC)}{\PSU(n)}
\]
and identify it via the map $m \PSU(n) \longmapsto m m^*$ with the space of positive definite Hermitian matrices, where $m^*$ denotes the conjugate transpose of $m \in  \PSL(n, \CC)$. Let $\Hn$ be the real vector space of Hermitian matrices and $\PHn$ the corresponding compact projective space.  For $A \in \Hn$ we denote the corresponding equivalence class in $\PHn$ by $[A]$. As $\Pn \subset\Hn$, the map $A \longmapsto [A]$ is a $\PSL(n, \CC)$-equivariant embedding and we define 
\[
      \PSat \coloneqq  \overline{i(\Pn)} \subset \PHn
\]
to be the \emph{Standard-Satake compactification}. 

Now let $\tau: G \lora \PSL(n, \CC)$ be a faithful projective representation of $G$. Via the map 
\begin{equation} \label{itau}
      \begin{aligned} 
	  i_\tau: X = G/K &\lora \Pn \\
	  gK &\longmapsto \tau(g)\tau(g)^*
      \end{aligned}
\end{equation}
we can embed $X$ into $\Pn$ as totally geodesic submanifold. There is a 1-to-1-correspondence between such embeddings and faithful projective representations  of $G$ into $\PSL(n, \CC)$ with the additional condition $\tau(\vartheta(g)) = (\tau(g)^*)^{-1}$ for all $g \in G$, where $\vartheta$ denotes the Cartan involution on $G$. With this we define 
\[
 \xsat \coloneqq \overline{i_\tau(X)} \subset \PSat
\]
as the \emph{generalized Satake compactification} of $X$ with respect to the representation $\tau$. By the action of $G$ on $\Pn$, $g \cdot A = \tau(g) A \tau(g)^*$ for $g \in G$ and $A \in \Pn$, the first embedding $i_\tau$ is $G$-equivariant and therefore $\xsat$ is a \mbox{$G$-compactification}, that is, the $G$-action on $X$ extends to a continuous action on $\xsat$.  
When $\tau$ is an irreducible representation, the compactification $\xsat$ is a classical Satake compactification, which has been introduced and described by Satake in \cite{Sat}. In general, when $\tau$ is reducible, the compactification $\xsat$ is a generalized Satake compactification as introduced and described in \cite{GuichardKasselWienhard}. 

Note that there are finitely many isomorphism classes of Satake compactifications, one associated to any proper subset $I\subset \Delta$, but infinitely many isomorphism classes of generalized Satake compactifications.

\subsection{The compactification of a flat in a generalized Satake compactification} 
We now compare the generalized Satake compactification with the horofunction compactification of $X$ with respect to an appropriate polyhedral $G$-invariant Finsler metric.

With the Cartan decomposition (see Lemma \ref{cartan} on page \pageref{cartan}) we can write $X = K \overline{A^+}.p_0$, and since $K$ is compact $\overline{X} =\overline{K \overline{A^+}.p_0}= K.\overline{\overline{A^+} p_0}$. Thus it is sufficient to show that the we have an W-equivariant homeomorphism between the closures of $A.p_0$  in the horofunction compactification and the generalized Satake compactification respectively. 

For the closure of the flat $F= A.p_0$ in the generalized Satake compactification we have the following:

\begin{thrm}[\cite{Ji-Art} Prop.4.1, \cite{GuichardKasselWienhard}]\label{thm:flat_Satake}
Let $\tau: G \to \PSL(n,\CC)$ be a faithful projective representation. Let $\mu_1, \ldots, \mu_k$ be the weights of $\tau$. 
Then the closure of the flat $A.p_0$ in the generalized Satake compactification $\xsat$ is $W$-equivariantly isomorphic to $\conv(2\mu_1, \ldots, 2\mu_k) \subset \aa^*$. 
\end{thrm}

\begin{rem}
 Because of the symmetry of the weights with respect to the Weyl chambers, the convex hull of all weights is the same as the convex hull of the Weyl-group orbit of $\chi_1, \cdots, \chi_l$, where $\chi_i$ are the highest weights of the irreducible components $\tau_i$ of $\tau$:  
 \[
  \conv(2 \mu_1, \ldots, 2\mu_k) = \conv\big( W(2 \chi_1), \cdots ,W(2 \chi_l)\big).
 \]
\end{rem}

\begin{ex}\label{Ex:SL3}
 Let us look at an example. Take  $X = \SL(3,\CC) / \SU(3)$ with the adjoint representation $\ad$ of $\gg$ which induces a representation on $G$. The weights of the adjoint representation are exactly the roots $\alpha_{ij} \in \aa^*$ with $1\leq i \neq j \leq 3$, where 
 \[
  \alpha_{ij} (H) = h_i - h_j 
 \]
 for any diagonal matrix $H = \diag(h_1, h_2, h_3) \in \aa$. The highest weight with respect to the positive Weyl chamber 
\[
      \aa^+ \coloneqq \left \{ \diag(h_1, h_2, h_3) \in \sl(3,\CC) \left | h_1 > h_2 > h_3, \ \sum_{i=1}^3 h_i = 0 \right. \right \}
\]
 is $\alpha_{13}$.

We identify $\aa$ with $\aa^*$ using the Killing form $\kappa$. 
Then the convex hull of the weights is shown in Figure \ref{fig:conv(mu)_SL(3,R)}.

 \begin{figure}[ht!]
	\centering
		\includegraphics{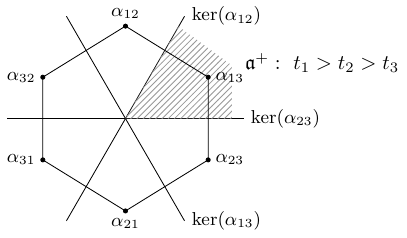}
	\caption{$\conv(\mu_1, \ldots, \mu_6)$ for the representation $\tau = \ad$ in $\aa^*$.} \label{fig:conv(mu)_SL(3,R)}
 \end{figure}

 By Theorem \ref{thm:flat_Satake} we know that the closure of the flat in $\xsat$ is homeomorphic to \mbox{$\conv(2\mu_1, \ldots, 2\mu_k) \subset \aa^*$.} If the highest weight is regular, like for $\tau= \ad$ as in this example, we obtain the maximal Satake compactification. \\
 To get the two minimal Satake compactifications, the highest weight has to lie on a singular direction, see Figure \ref{Fig:triangles} for a picture. The representations here are the standard and the dual standard representation. 
 
 \begin{figure}[h!]
  \centering
 	\includegraphics{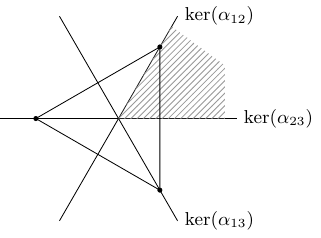}
 	\includegraphics{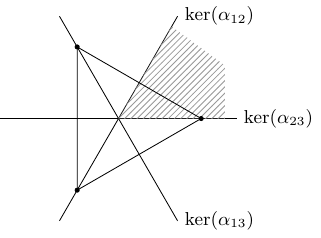}
 \caption{These two convex hulls correspond to the standard and the dual standard representations.}
 \label{Fig:triangles}
 \end{figure}
 
 If we now take the convex hull of these two triangles, we again obtain a hexagon but now with its vertices on the singular directions, see Figure \ref{Fig:hexagon_singular}. This compactification of the flat corresponds to a generalized Satake compactification associated to the direct sum of the standard and the dual standard representation. It is the same as the polyhedral compactification of the flat with respect to the polyhedral decomposition of $\aa$ with respect to the Weyl chambers.
 
  \begin{figure}[h!]
 \centering
	  \includegraphics{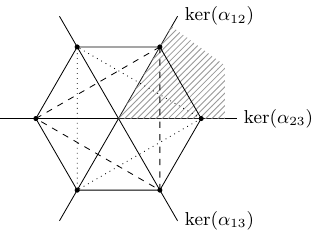}
  \caption{The convex hull of the two balls above give a hexagon with vertices on the singular directions.}
  \label{Fig:hexagon_singular}
 \end{figure}

 \end{ex}
 
\begin{prop} \label{prop:comp_flat}
 Let $X = G / K$ be a symmetric space of non-compact type. 
 Let $\tau$ be a faithful projective representation of $G$, and  $\mu_1, \ldots, \mu_n$ its weights. 
Let $D \coloneqq \conv(2\mu_1, \ldots, 2\mu_n) \subset \aa^*$. 
 Let $B = -D^\circ$ the dual closed convex set  in the maximal abelian subalgebra $\aa \subset \pp \subset \gg$. 
 Then the closure of the flat $A.p_0$ in the generalized Satake compactification is $W$-equivariantly isomorphic to the closure of the flat $A.p_0$ 
 in the horofunction compactification of $X$ with respect to the Finsler metric defined by $B$. 
 \end{prop}
\begin{proof}
By Theorem~\ref{thm:embedding flat in X}, it suffices to compare the closure of $A^+.p_0$  in the generalized Satake compactifications with the closure of $A.p_0$ in the flat compactification of $A.p_0$ with respect to the norm defined by $B$. By Theorem~\ref{thm:flat_Satake} and Theorem~\ref{thrm:flatcompactification}, both are W-equivariantly homeomorphic to the closed convex $\conv(2\mu_1, \ldots, 2\mu_n) = D = -B^\circ$. 

Note that in Theorem~\ref{thm:flat_Satake} and Theorem~\ref{thrm:flatcompactification}, the identification of the closure of $A.p_0$ with $D$ relies on a moment map, and a direct comparison shows that a sequence $H_n \in \aa$  converges in the Satake compactification $\xsat$ if and only if it converges in the horofunction compactification $\xhor$ with respect to the $G$-invariant Finsler metric defined by $B$.
\end{proof}

\begin{thrm} \label{thm:comp}
 Let $X = G / K$ be a symmetric space of non-compact type. 
 Let $\tau$ be a faithful projective representation of $G$ and $\mu_1, \ldots, \mu_n$ its weights. Let $D \coloneqq \conv(\mu_1, \ldots, \mu_n) $. 
 Let $B = -D^\circ$ define a unit ball in the maximal abelian subalgebra $\aa \subset \pp \subset \gg$. 
 Then the generalized Satake compactification $\xsat$ is $G$-equivariantly isomorphic to the horofunction compactification of $X$ with respect to the Finsler metric defined by $B$. 
\end{thrm}

\begin{proof}
We show that   a sequence converges in the generalized Satake compactification $\xsat$ if and only if it converges in the horofunction compactification $\xhor$ with respect to the $G$-invariant Finsler metric defined by $B$.
Let $x_n \in X$ be a sequence. Then we can write $ x_n  = k_n \cdot a_n p_0$, where $k_n \in K$ and $a_n \in \overline{A^+}$ is uniquely determined. 
Up to passing to a subsequence we can assume that $x_n$ converges in $\xsat$  and that $k_n$ converges to an element $k\in K$. 
Therefore Theorem~\ref{thm:comp} is a consequence of Proposition~\ref{prop:comp_flat}. 
\end{proof}

\begin{rem}
Note that Theorem~\ref{thm:comp} describes explicitly the convex unit ball of the Finsler metric which induces the horofunction compactification realizing the (generalized) Satake compactifications. 
For classical Satake compactifications the convex $D$ (and hence also the unit ball $B$) has a particularly simple description as it is just the convex hull of the Weyl group orbit of the highest weight vector of $\tau$. In order to obtain the Satake compactification determined by a subset $I \subset \Delta$ one has to choose a representation $\tau$, whose highest weight vector has support equal to $I$. 
\end{rem}

\begin{ex}
      We consider $X = \SL(4,\CC)/\SU(4)$ with the same notations as in Example \ref{Ex:SL3} above.  Let again $\tau = \ad$ be the considered representation.  Then the highest weight with respect to the positive Weyl chamber 
      \[
	  \aa^+ \coloneqq  \left \{\diag(t_1, \ldots, t_4) \in \sl(4,\CC) \left | t_1 > \ldots > t_4,  \ \sum_{i = 1}^{4} t_i = 0 \right. \right \}
      \]
      is $\mu_\tau = \alpha_{14}$. Let $M_\tau \in \aa$ be the element corresponding to $\mu_\tau$ by identifying $\aa$ and $\aa^*$ with the Killing form. Note that as $\alpha_{23}(M_\tau) = 0$, $M_\tau$ lies on a Weyl chamber wall. The Weyl chamber system is shown in Figure \ref{fig:Weyl_2zusammen_SL(4,R)}. The picture on the left illustrates the structure of the Weyl chamber walls while the one on the right shows the positive Weyl chamber we chose.

	\begin{figure}[h!]
	  \centering
		  \includegraphics{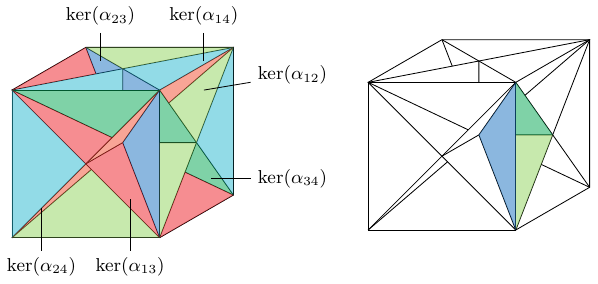}
	  \caption{The Weyl chamber system of $\SL(4,\RR)$ with all Weyl chamber walls (left) and the positive Weyl chamber we chose (right).}\label{fig:Weyl_2zusammen_SL(4,R)}
      \end{figure}

      \begin{figure}[h!]
	  \centering
		  \includegraphics{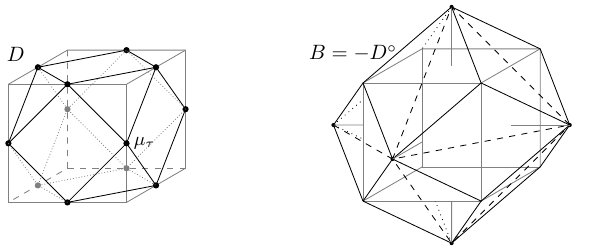}
	  \caption{$D=\conv(W(\mu_\tau))$ and $B=-D^\circ$ for $\tau=\ad$.}\label{Fig:SL(4,R)_BD_adjoint}
      \end{figure}

      The convex hull $D$ of the weights is a regular polyhedral ball with 12 vertices and 14 maximal dimensional faces. Accordingly, the unit ball \linebreak $B = -D^\circ$ has 14 vertices and 12 maximal dimensional faces, a picture of both is given in Figure \ref{Fig:SL(4,R)_BD_adjoint}. The dashed lines in the right picture are to indicate that always two triangles together form a rhomb. If we chose the Finsler metric corresponding to $B$ as unit ball, we obtain a horofunction compactification of the flat which is isomorphic to the Satake compactification with respect to $\tau=\ad$. 

      Let us now consider other representations. If the representative $M_\tau$ of the highest weight lies completely inside of $\aa^+$ we get $D$ and $B$ as shown in Figure \ref{Fig:SL(4,R)_BD_regular}. The polyhedron $D=\conv(W(\mu_\tau))$ is then called the permutohedron of dimension $3$. More generally, if $\tau$ is a faithful representation of $\SL(n-1,\RR)$ with regular $M_\tau$, then the polyhedron $D=\conv(W(\mu_\tau))$ is the $(n-1)$-dimensional permutohedron.

      \begin{figure}[h!]
	  \centering
		  \includegraphics{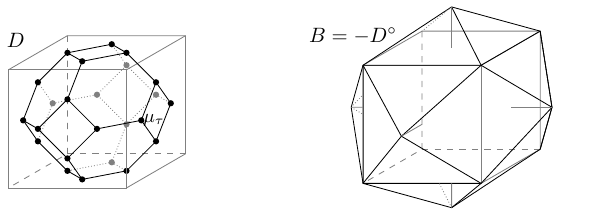}
	  \caption{$D$ and $B =-D^\circ$ for a representation with highest weight in a regular direction.} \label{Fig:SL(4,R)_BD_regular}
      \end{figure}

      On the other hand if $M_\tau$ lies in more than one Weyl chamber wall, the convex hull $D$ of the weights and the unit ball $B$ of the Finsler norm are like shown in Figure \ref{Fig:SL(4,R)_BD_singular} or a rotated version of it, depending on which pair of Weyl chamber walls $M_\tau$ lies.  

      \begin{figure}[h!]
	  \centering
		  \includegraphics{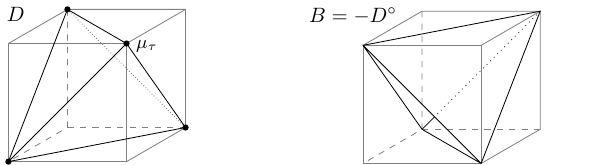}
	  \caption{$D$ and $B =-D^\circ$ for a representation with highest weight in a singular direction.} \label{Fig:SL(4,R)_BD_singular}
      \end{figure}

\end{ex}

\subsection{Dual generalized Satake compactifications}
The realization of generalized Satake compactifications as horofunction compactifications for polyhedral $G$-invariant Finsler metrics allows us to define the dual generalized Satake compactification ${\overline{X}_{\tau}^S}^*$ . 
\begin{defi}
Let $\tau: G \to \PSL(n,\CC)$ be a faithful projective representations and $\overline{X}_{\tau}^S$ the associated generalized Satake compactification. The dual generalized Satake compactification ${\overline{X}_{\tau}^S}^*$ is defined to be the horofunction compactification of $X$ with respect to the polyhedral $G$-invariant Finsler metric defined by the unit ball 
$B= D = \conv(\mu_1, \ldots, \mu_k)$.
\end{defi}

\begin{quest}
Is there a geometric way to interpret the duality between $\overline{X}_{\tau}^S$ and ${\overline{X}_{\tau}^S}^*$?
 \end{quest}

There are many polyhedral $G$-invariant Finsler metrics which are not related to generalized Satake compactifications, and even more $G$-invariant Finsler metrics which are not polyhedral. Since any Weyl group invariant convex set containing $0$ defines a $G$-invariant Finsler metric, it is very natural to ask whether natural operations on convex sets extend to natural operations on the corresponding horofunction compactifications.


\bigskip

\begin{flushright}
Thomas Haettel \\
IMAG, Univ Montpellier, CNRS, Montpellier, France\\
thomas.haettel@umontpellier.fr\\
\ \\ \ \\
Anna-Sofie Schilling \\
Mathematisches Institut \\
Ruprecht-Karls Universit\"at Heidelberg\\
Im Neuenheimer Feld 205, 69120 Heidelberg, Germany \\
aschilling@mathi.uni-heidelberg.de \\
\ \\ \ \\
Cormac Walsh \\
Inria, CMAP, Ecole polytechnique, CNRS, Université Paris-Saclay, 91128,
Palaiseau, France. \\
 cormac.walsh@inria.fr
\ \\ \ \\
Anna Wienhard \\
Mathematisches Institut\\
Ruprecht-Karls Universit\"at Heidelberg\\
Im Neuenheimer Feld~205, 69120 Heidelberg, Germany\\

\ \\
HITS gGmbH Heidelberg Institute for Theoretical Studies\\
Schloss-Wolfs\-brunnen\-weg 35, 69118 Heidelberg, Germany\\
wienhard@uni-heidelberg.de \\
\ \\

\end{flushright}

\begin{thebibliography}{WX}
 \bibitem[AGW]{AGW}
Marianne Akian, St{\'e}phane Gaubert, and Cormac Walsh.
\newblock {The max-plus {M}artin boundary}.
\newblock {Documenta Mathematica 14 (2009)}: 195-240. <http://eudml.org/doc/225846>.


\bibitem[Bee93]{Beer}
Gerald Beer.
\newblock {\em {Topologies on closed and closed convex sets}}, volume 268 of   {\em {Mathematics and its Applications}}.
\newblock Kluwer Academic Publishers Group, Dordrecht, 1993.

\bibitem[BJ06]{BJ}
Armand Borel and Lizhen Ji.
\newblock {\em {Compactifications of symmetric and locally symmetric spaces}}.
\newblock {Mathematics: Theory \& Applications}. Birkh{\"a}user Boston, Inc.,
  Boston, MA, 2006.

\bibitem[Bou59]{Bourbaki_integration}
Nicolas Bourbaki.
\newblock {\em {Intégration. Chapitre 8}}, in  {\em {\'{E}léments de mathématiques}}.
\newblock Hermann, 1959.

\bibitem[Brill]{Brill}
B. Brill.
\newblock{{\em A family of compactifications of affine buildings. (Eine Familie von Kompaktifizierungen affiner Geb{\"a}ude.)}}. 
\newblock {Berlin: Logos Verlag}. Frankfurt am Main: Fachbereich Mathematik (Dissertation), 2006.

\bibitem[CKS]{CKS}
C.Ciobotaru, L. Kramer, P. Schwer
\newblock{private communication}. 

  
\bibitem[FF04I]{FriedFreiI}
Shmuel Friedland and Pedro J. Freitas.
\newblock{p-Metrics on $GL(n,\CC)/U_n$ and their Busemann compactifications}.
\newblock{{\em Linear Algebra and its Applications}}, volume 376, 1-18, 2004.
\newblock{http://www.sciencedirect.com/science/article/pii/S002437950300661X}.
  

\bibitem[FF04II]{FriedFreiII}
Shmuel Friedland and  Pedro J. Freitas.
\newblock{Revisiting the Siegel upper half plane I}.
\newblock{{\em Linear Algebra and its Applications}}, volume 376, 19 - 44, 2004.
\newblock{http://www.sciencedirect.com/science/article/pii/S0024379503006621}.
  
  
\bibitem[Hal03]{Hall}
Brian~C. Hall.
\newblock {\em {Lie groups, {L}ie algebras, and representations}}, volume 222
  of {\em {Graduate Texts in Mathematics}}.
\newblock Springer-Verlag, New York, 2003.
\newblock An elementary introduction.

\bibitem[Hel78]{Helg}
Sigurdur Helgason.
\newblock {\em {Differential geometry, {L}ie groups, and symmetric spaces}},
  volume~80 of {\em {Pure and Applied Mathematics}}.
\newblock Academic Press Inc. [Harcourt Brace Jovanovich Publishers], New York,
  1978.

\bibitem[GJT98]{GJT}
Yves Guivarc'h, Lizhen Ji, and J. C. Taylor, 
\newblock {\em {Compactifications of symmetric spaces}},
  volume~156 of {\em {Progress in Mathematics}}.
\newblock Birkh\"auser Boston.

\bibitem[GKW17]{GuichardKasselWienhard}
Olivier Guichard, Fanny Kassel, and Anna Wienhard. 
\newblock{Tameness of Riemannian locally symmetric spaces arising from Anosov representations}
\newblock{{\em ArXiv e-prints}}, August 2015.
\newblock arXiv:1508.04759, under revision.

\bibitem[Ji97]{Ji-Art}
Lizhen Ji.
\newblock {Satake and {M}artin compactifications of symmetric spaces are
  topological balls}.
\newblock {\em Math. Res. Lett.}, 4(1):79--89, 1997.

\bibitem[JS16]{JS}
Lizhen Ji and Anna-Sofie Schilling.
\newblock{Polyhedral Horofunction Compactification as Polyhedral Ball}
\newblock{{\em ArXiv e-prints}}, August 2016.
\newblock arXiv:1607.00564v2.

\bibitem[KL16]{KL}
Michael Kapovich and Bernhard Leeb.
\newblock{Finsler bordifications of symmetric and certain locally symmetric spaces}
\newblock{{\em ArXiv e-prints}}, March 2016.
\newblock arXiv:1505.03593v6.

\bibitem[KMN06]{KarlssonMetzNoskov}
Anders Karlsson, Volker Metz, and Gennady A. Noskov. 
\newblock{Horoballs in simplices and Minkowski spaces }
\newblock{{\em  Int. J. Math. Math. Sci.}}, , 1-20, 2006.

\bibitem[Pa]{Parreau}
Anne Parreau.
\newblock {private communication}.

\bibitem[Pla95]{Plan}
Pierre Planche.
\newblock {\em {G{\'e}om{\'e}trie de Finsler sur les Espaces Sym{\'e}triques}}.
\newblock PhD thesis, Universit{\'e} de Gen{\`e}ve, 1995.

\bibitem[Rie02]{Rief}
Marc~A. Rieffel.
\newblock {Group {$C^*$}-algebras as compact quantum metric spaces}.
\newblock {\em Doc. Math.}, 7:605--651 (electronic), 2002.

\bibitem[Sat60]{Sat}
Ichir{\^o} Satake.
\newblock {On representations and compactifications of symmetric {R}iemannian spaces}.
\newblock {\em Ann. of Math. (2)}, 71:77--110, 1960.

\bibitem[S13]{Schilling}
Anna-Sofie Schilling. 
\newblock {Busemann Compactification of Finite-Dimensional
Normed Spaces and of Symmetric Spaces}.
\newblock Diplomarbeit Uni.\ Heidelberg, 2013.

\bibitem[Wal07]{WalII}
Cormac Walsh.
\newblock {The horofunction boundary of finite-dimensional normed spaces}.
\newblock {\em Math. Proc. Cambridge Philos. Soc.}, 142(3):497--507, 2007.

\bibitem[{Wal}10]{WalI}
Cormac Walsh.
\newblock {The horoboundary and isometry group of Thurston's Lipschitz metric}.
\newblock In {\em Handbook of {T}eichm\"uller theory. {V}ol. {IV}}, volume~19
  of {\em IRMA Lect. Math. Theor. Phys.}, pages 327--353. Eur. Math. Soc.,
  Z\"urich, 2014.


\end{thebibliography}
\end{document}